\documentclass[11pt,a4paper]{amsart}
\usepackage{amssymb,amsmath,amsthm}
\usepackage{tikz}
\usepackage{tikz-cd}
\usepackage{enumitem}
\usepackage{amsaddr}
\usepackage{etoolbox}
\usepackage{upgreek}\usetikzlibrary{matrix}
\usepackage[utf8]{inputenc}
\usepackage{hyperref}
\usepackage{thmtools,thm-restate}

\newcommand{\mylabel}[2]{#2\def\@currentlabel{#2}\label{#1}}

\newcommand{\s}{\mathsf}
\newcommand{\f}{\mathfrak}
\newcommand{\C}{\mathcal}

\newcommand{\ZZ}{\s{ZZ(A)}}
\newcommand{\ZZval}{\s{ZZ}_{\mathrm{val}}(\s A)}
\newcommand{\crowns}{\s{Crowns(A)}}
\newcommand{\crownsval}{\s{Crowns}_{\mathrm{val}}(\s A)}

\newcommand{\pcrownsval}{\s{PCrowns}_{\mathrm{val}}(\s A)}
\newcommand{\crownsspec}{\s{Crowns}_{\mathrm{spec}}(\s A)}
\newcommand{\pcrownsspec}{\s{PCrowns}_{\mathrm{spec}}(\s A)}

\newcommand{\cq}[1]{\overline{\C{Q}}_{#1}}
\newcommand{\cs}{\overline{\C{S}}}
\newcommand{\ct}{\overline{\C{T}}}

\newcommand{\wst}[1]{\C W_{\s{St}}(#1)}
\newcommand{\winvst}[1]{\overline{\C W}_{\s{St}}(#1)}
\newcommand{\wba}[1]{\C W_{\s{Ba}}(#1)}
\newcommand{\winvba}[1]{\overline{\C W}_{\s{Ba}}(#1)}

\newcommand{\cyc}{\s{Cyc}(\Lambda)}
\newcommand{\cycsp}{\mathsf{Cyc}^{\mathrm{Sp}}(\Lambda)}

\newcommand{\INF}[1]{\,^\infty{#1}^\infty}

\newcommand{\marst}{\s{St}_{\bullet}(\Lambda)}

\newcommand{\FacZ}[1]{\mathrm{Fac}_\mathrm{fin}^\mathbb{Z}(#1)}
\newcommand{\ImZ}[1]{\mathrm{Im}_\mathrm{fin}^\mathbb{Z}(#1)}

\begin{document}
\newcounter{thmcount}
\setcounter{thmcount}{0}
\newtheorem{theom}[thmcount]{Theorem}
\renewcommand*{\thetheom}{\Alph{theom}}

\newtheorem{defn}{Definition}[section]
\newtheorem{definitions}[defn]{Definitions}
\newtheorem{lem}[defn]{Lemma}
\newtheorem{construction}[defn]{Construction}
\newtheorem{prop}[defn]{Proposition}
\newtheorem*{prop*}{Proposition}
\newtheorem{thm}[defn]{Theorem}
\newtheorem{cor}[defn]{Corollary}
\newtheorem{claim}{Claim}[defn]
\newtheorem*{claim*}{Claim}
\newtheorem{algo}[defn]{Algorithm}
\theoremstyle{remark}
\newtheorem{remarks}[defn]{Remarks}
\theoremstyle{remark}
\newtheorem{notation}[defn]{Notation}
\theoremstyle{remark}
\newtheorem{exmp}[defn]{Example}
\AtEndEnvironment{exmp}{\hfill$\Diamond$}
\theoremstyle{remark}
\newtheorem{examples}[defn]{Examples}
\theoremstyle{remark}
\newtheorem{dgram}[defn]{Diagram}
\theoremstyle{remark}
\newtheorem{fact}[defn]{Fact}
\theoremstyle{remark}
\newtheorem{illust}[defn]{Illustration}
\theoremstyle{remark}
\newtheorem{que}[defn]{Question}
\theoremstyle{remark}
\newtheorem{rmk}[defn]{Remark}
\theoremstyle{remark}
\newtheorem{ques}[defn]{Question}
\numberwithin{equation}{section}

\author{Annoy Sengupta and Amit Kuber}
\address{Department of Mathematics and Statistics\\Indian Institute of Technology, Kanpur\\ Uttar Pradesh, India}
\email{annoysgp20@iitk.ac.in, askuber@iitk.ac.in}

\title[Characterisation of band bricks]{Characterisation of band bricks over certain string algebras and a variant of perfectly clustering words}
\date{\today}

\maketitle

\begin{abstract}
Generalising a recent work of Dequ\^ene et al. on the connection between perfectly clustering words and band bricks over a particular family of gentle algebras, we characterise band bricks over string algebras whose underlying quiver is acyclic in terms of \emph{weakly perfectly clustering pairs of words}---a variant of perfectly clustering words. As a consequence, we characterise band semibricks over all such algebras. Furthermore, the combination of our result and a result of Mousavand and Paquette provides an algorithm to determine whether such a string algebra is brick-infinite.
\end{abstract}

\section{Introduction}
Throughout the paper we will assume that $\C K$ is an algebraically closed field. Drozd \cite{drozd1977matrix} and Crawley-Boevey \cite{crawley1988tame} classified the class of finite-dimensional associative algebras over $\C K$ into finite, tame and wild representation types based on the complexity of the problem of classification of $\Lambda$-$\mathrm{mod}$---the category of its finitely generated(=finite-dimensional) (left) representations(=modules). Recall that a finite representation type algebra admits finitely many (isomorphism classes of) indecomposable modules, a tame representation type algebra admits infinitely many (isomorphism classes of) indecomposable modules but, for a given total dimension $d$, all but finitely many modules occur in at most finitely many one-parameter families, whereas the category $\Lambda$-$\mathrm{mod}$ for a wild representation type algebra $\Lambda$ contains the category of finite-dimensional modules over any finite-dimensional algebra. However, since it is hard to classify $\Lambda$-$\mathrm{mod}$, it is generally difficult to answer whether $\Lambda$ is tame or wild. Nevertheless, there are some tame or wild algebras for which one can get a good description of its bricks, where a brick is a module whose endomorphism algebra is isomorphic to the ground field $\C K$, and hence indecomposable. More precisely, there are some tame or wild algebras that are brick-finite or brick-tame, where the latter two classes have obvious definitions. For example, a representation-infinite local algebra is tame but admits a unique brick. See \cite{bodnarchuk2010one} for an example of a wild algebra that is brick-tame. Since brick-finite algebras are a generalisation/relaxation of representation-finite algebras, the problem of characterising whether a finite-dimensional algebra is brick-finite or not is of particular interest along with the classification of all such bricks.

A recent work by Dequ\^ene et al. \cite{dequêne2023generalization} gave a surprising as well as interesting connection between the two seemingly different fields, viz. combinatorics on words and the study of bricks over a particular family $\{\Lambda_N\mid N\geq2\}$ (see Figure \ref{fig:lambda_n1}) of gentle algebras.
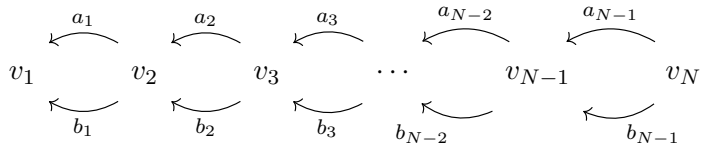
\begin{figure}[h]
    \centering
    \[\begin{tikzcd}
v_1 & v_2 \arrow[l, "b_1", bend left, shift left=2] \arrow[l, "a_1"', bend right, shift right=2] & v_3 \arrow[l, "b_2", bend left, shift left=2] \arrow[l, "a_2"', bend right, shift right=2] & \cdots \arrow[l, "a_3"', bend right, shift right=2] \arrow[l, "b_3", bend left, shift left=2] & v_{N-1} \arrow[l, "a_{N-2}"', bend right, shift right=2] \arrow[l, "b_{N-2}", bend left, shift left=2] & v_N \arrow[l, "b_{N-1}", bend left, shift left=2] \arrow[l, "a_{N-1}"', bend right, shift right=2]
\end{tikzcd}\]
    \caption{$\Lambda_N$ with $\rho=\{a_1b_2,a_2b_3,\cdots,a_{N-2}b_{N-1},b_1a_2,b_2a_3,\cdots,b_{N-2}a_{N-1}\}$}
    \label{fig:lambda_n1}
\end{figure}
In the field of combinatorics of words, Christoffel words are a well-studied family of words on a two-letter alphabet. There are many definitions and equivalent characterizations of those words, but one may refer to \cite{berstel2009combinatorics} for a geometric definition that involves certain paths in the integer lattice, and another definition that involves reading edge-labellings of the Cayley graph of a certain group. The geometric definition splits Christoffel words into two types---upper Christoffel words and lower Christoffel words. The lower Christoffel words were generalised in \cite{simpson2008words} to alphabets of any size; such words are called \emph{perfectly clustering words}. For a fixed $N\geq2$, the authors of \cite{dequêne2023generalization} associated to each word $\s w$ over the alphabet $\s A:=\{2,\cdots, N\}$ a band $\varphi(\s w)$ for $\Lambda_N$, and gave a necessary and sufficient condition for a primitive(=not a $k$-fold concatenation of a word for any $k\geq2$) word $\s w$ to be perfectly clustering.

\begin{restatable}{theom}{theirmainthm}\label{thm: their main thm}
\cite[Theorem~5.12]{dequêne2023generalization} Let $N\geq 2$ and $\s A:=\{2,\cdots,N\}$ be an alphabet linearly ordered with the usual ordering. Then a primitive $\s A$-word $\s w$ is perfectly clustering if and only if the band module $B(\varphi(\s w),1,\lambda)$ for some (equivalently, any) $\lambda\in\C K^\times$ for $\Lambda_N$ is a brick.
\end{restatable}

Although our work is based on \cite{dequêne2023generalization}, the motivation for our work differs from theirs. We generalise the above theorem to give a characterisation of band-bricks for string algebras whose associated (Gabriel) quiver is acyclic i.e. there is no directed cycle (Theorem \ref{thm: main}). As applications of our result, we also characterise band semibricks (Corollary \ref{cor: semi-brick criterion}), which are finite direct sums of band-bricks where none of the distinct pair admits a non-zero morphism between them, and using a result of Mousavand and Paquette, we give an algorithm (Corollary \ref{cor: algorithm brick inf}) to determine whether or not a string algebra whose underlying quiver is acyclic is brick-infinite.

The Burrows-Wheeler transformation (Definition \ref{defn: pcw}) of a word is the sequence of the last letters of the cyclic permutations of the word arranged in the lexicographic order. A particular map $\Phi$, defined by Gessel and Reutenauer \cite{gessel1993counting}, from the set of words to the multiset of conjugacy classes of primitive words, gives an inverse map of the Burrows-Wheeler transformation for perfectly clustering words. Lapointe \cite{lapointe2020combinatoire} gave a conjecture about an upper bound of distinct lengths of perfectly clustering words appearing in $\Phi(n^{\alpha_n}\cdots2^{\alpha_2})$ \cite[Conjecture~1.2]{dequêne2023generalization}). The authors in \cite{dequêne2023generalization} proved a stronger version of this conjecture \cite[Corollary~1.3]{dequêne2023generalization} using Theorem \ref{thm: their main thm}. On the other hand, our motivation for extending their result was to determine (in finite time) whether a given band module is a brick. Our work also differs from theirs in the methods used. They used geometric models to study gentle algebras and their representations (developed in \cite{assem2010gentle}, \cite{opper2018geometric}, \cite{palu2019non}, \cite{baur2021geometric}), whereas we use combinatorial descriptions of representations and morphisms between them (developed in \cite{gel1968indecomposable}, \cite{butler1987auslander}, \cite{crawley1989maps},\cite{Krause1991MapsBT}).

In this paper, the generalisation of Theorem \ref{thm: their main thm} is done in two steps. Firstly, we note that the bands considered had a specific configuration. The path that those bands traversed involved returning to vertex $1$ after reaching any of the remaining vertices. Therefore,  for any band that did not have this configuration, it is not possible to determine whether the corresponding band module is a brick. In order to encompass all bands, we associate each string with a word formed by the vertices where the string changes its nature from direct to inverse or vice-versa. The new words formed are \emph{zigzags} (Definition \ref{defn: zigzagcrown}). For bands, the corresponding ``cyclic'' words are called \emph{crowns} (Definition \ref{defn: zigzagcrown}), which are zigzags whose every cyclic permutation is also a zigzag. This association is consistent with the old one as the new words corresponding to the old bands are perfectly clustering if and only if the old words are so (Proposition \ref{prop: new association consistent}). A slight modification of perfectly clustering words yields \emph{weakly perfectly clustering words} (Definition \ref{defn: weak perf clus}), which is used in the characterisation of all band bricks over $\Lambda_N$.

\begin{restatable}{theom}{brickiffwpcc}\label{thm: brick iff WPCC}
Let $\widetilde{\s A}:=\{1,2,\cdots,N\}$ with the usual ordering be a linearly ordered alphabet. A primitive $\widetilde{\s A}$-crown $\s w$ is weakly perfectly clustering if and only if the band module $B(\widetilde{\varphi}(\s w),1,\lambda)$ for some (equivalently, any) $\lambda\in\C K^\times$ for $\Lambda_N$ is a brick.
\end{restatable}

Using the rule that if there is a direct arrow from $v_i$ to $v_j$ then $v_i>v_j$, we note that the set of vertices will form a partially ordered set instead of a linearly ordered set for string algebras whose underlying quiver is acyclic. The definition of a weakly perfectly clustering crown over a linearly ordered alphabet involves the comparison of certain intervals of the linearly ordered set. However, an interval of a partially ordered set is not necessarily a linear order. It is realised (see Example \ref{exmp: motivation of traced poset}) that it is not actually the intervals that we compare, but it is the appropriate path that the band traverses. This motivated us to ``trace" appropriate paths or linear orders in the partially ordered set based on whether we can reach from one point to another using a direct string or an inverse string. Thus we define a \emph{finite traced partially ordered set} (Definition \ref{defn: traced poset}). Next, we notice that it is actually the zero-length strings and not the vertices that we need to note down while constructing a word out of a string because both zero-length strings $1_{(v,\pm1)}$ corresponding to a single vertex $v$ may appear simultaneously in a string for arbitrary string algebras whose underlying quiver is acyclic, contrary to that in $\Lambda_N$ (Example \ref{exmp: motivation of cov quiv}). Therefore, we obtain a \emph{covering quiver} (Definition \ref{defn: cov quiv}) from a string algebra whose vertices are zero-length strings and arrows are direct or inverse syllables. The covering quiver allows us to get a sense of the direction of the strings. Also, this step is an intermediate step for obtaining a finite traced partially ordered set from such a string algebra. Now the vertices of the traced partially ordered set are the vertices of the covering quiver(=zero-length strings) and the ordering is defined usually, i.e. if there is a direct (resp. inverse) arrow from $v$ to $v'$ then $v>v'$ (resp. $v<v'$). It is important to observe that the hypothesis that the underlying quiver of the string algebra is acyclic is required to prove that this order is a partial order (Proposition \ref{prop: poset from cov quiv}).

The finite traced partially ordered set forms an alphabet for appropriate words that would be correspondents of strings. However, as opposed to $\Lambda_N$, not all zigzags will correspond to strings; the ones that do are called \emph{valid zigzags} (Definition \ref{defn: valid zigzag}). Similarly, the correspondent $\wba{\f b}$ of a band $\f b$ is called a \emph{primitive special crown} (Definition \ref{defn: spec crown}). Proposition \ref{prop:bijections} documents all such one-to-one correspondences. Once both sides as well as the bridges between them are established, a technical result (Lemma \ref{lem: nontrivial morphism between bands criterion in terms of WPCC}), that forms the heart of the argument of the main result of the paper (Theorem \ref{thm: main}), stated below, gives a criterion for the existence of a non-trivial morphism between certain band modules.

\begin{restatable}{theom}{main}\label{thm: main}
Let $\Lambda$ be a string algebra whose underlying quiver is acyclic. Given a band $\f b$ for $\Lambda$, $l\in\mathbb N^+$ and $\lambda\in\C K^\times$, the band module $B(\f b,l,\lambda)$ is a brick if and only if $l=1$ and the pairs $(\wba{\f b},\wba{\f b})$ and $(\wba{\f b},\wba{\f b^{-1}})$ of crowns are weakly perfectly clustering.
\end{restatable}

The paper is organised as follows. The background and important notations about string algebras are set up in \S~\ref{sec: prel of str alg}. The work of \cite{dequêne2023generalization} is briefly discussed in \S~\ref{sec: band bricks in lamdaN} followed by a generalisation that characterises all the band bricks of $\Lambda_N$. In \S~\ref{sec: fin traced poset}, a finite traced partially ordered set is defined which forms an alphabet for certain words that would be correspondents of strings. The section also contains the definition of a weakly perfectly clustering pair of crown (Definition \ref{defn: weak perf clus2}), a generalisation of weakly perfectly clustering word, which is used in Lemma \ref{lem: nontrivial morphism between bands criterion in terms of WPCC}. The recipe for obtaining a finite traced partially ordered set from a string algebra is contained in \S~\ref{sec:obtain traced poset}. This section also contains the concept of a covering quiver, which is obtained from a string algebra as an intermediate step before obtaining the finite traced poset. Once we obtain a finite traced poset, it is natural to ask whether we can recover the string algebra from it. An affirmative answer to this question is given in \S~\ref{sec:construct string alg}. A reader who is only interested in the main result may skip this short section. Finally, \S~\ref{sec: main} establishes the connection between the two sides---strings and bands over string algebra, and certain words over a finite traced partially ordered set---and contains important results of this paper. Two applications of the main result are mentioned in \S~\ref{sec: appl}. The fact that Theorem \ref{thm: main} is applicable to all gentle algebras is established in \S~\ref{sec: future dir}, along with a future direction.

\subsection*{Acknowledgements}
The authors would like to thank Debajyoti Deb for useful preliminary discussions.

\subsection*{Funding}
The first author thanks the \emph{Council of Scientific and Industrial Research (CSIR)} India - Research Grant No. 09/092(1090)/2021-EMR-I for the financial support.

\section{Preliminaries of String Algebras}\label{sec: prel of str alg}
In this section, we set up notations and outline some background for string algebras relevant to this paper. The content and presentation of this section mostly follows \cite{LakThe16}. We begin by defining some combinatorial objects known as (finite and infinite) strings and bands consisting of words made up of letters arising out of arrows in the quiver. These combinatorial objects are the key to constructing indecomposable modules known as string modules and band modules. These two classes of indecomposable modules exhaust all the finite-dimensional indecomposable modules over a string algebra; a proof can be found in \cite{butler1987auslander}, which is motivated from \cite{gel1968indecomposable}.

\begin{defn}\label{defn:string alg}
A \emph{string algebra} is a monomial algebra $\C KQ/\langle\rho\rangle$, where $Q=(Q_0,Q_1,s,t)$ is a finite quiver satisfying the following properties.
\begin{enumerate}[label=(\Roman*)]
    \item\label{indegoutdeg} For every $v\in Q_0$, there are at most two arrows $\alpha,\beta\in Q_1$ such that $s(\alpha)=s(\beta)=v$, and at most two arrows $\alpha',\beta'\in Q_1$ such that $t(\alpha')=t(\beta')=v$.
    \item\label{at most one arrow} For every $\alpha\in Q_1$, there is at most one arrow $\beta\in Q_1$ such that $\alpha\beta\notin\rho$ and at most one arrow $\beta'$ such that $\beta'\alpha\notin\rho$.
    \item\label{admissible} There exists $M\geq 0$ such that $|p|\leq M$ for every path $p$ having no subpath in $\rho$.
\end{enumerate}
If all the relations in $\rho$ is of length 2 and Condition \ref{at most one arrow} is replaced by the following conditions, then $\C KQ/\langle\rho\rangle$ is said to be a \emph{gentle algebra}.
\begin{enumerate}[label=(II\alph*)]
    \item For every $\alpha\in Q_1$, there is at most one arrow $\beta\in Q_1$ such that $\alpha\beta\notin\rho$ and at most one arrow $\beta'\in Q_1$ such that $\alpha\beta'\in\rho$.
    \item For every $\alpha\in Q_1$, there is at most one arrow $\gamma\in Q_1$ such that $\gamma\alpha\notin\rho$ and at most one arrow $\gamma'\in Q_1$ such that $\gamma'\alpha\in\rho$.
\end{enumerate}
\end{defn}
If $\C KQ/\langle\rho\rangle$ is a string algebra then we also say that $(Q,\rho)$ is a presentation of a string algebra and we refer $Q$ as the underlying quiver of $\Lambda$.

For each string algebra, it is useful to choose and fix maps $\varsigma,\varepsilon:Q_1\to\{-1,1\}$ with the following properties.
\begin{enumerate}[label=(\alph*)]
    \item If $\alpha,\beta\in Q_1$ such that $\alpha\neq\beta$ with $s(\alpha)=s(\beta)$, then $\varsigma(\alpha)=-\varsigma(\beta)$.
    
    \item If $\alpha,\beta\in Q_1$ such that $\alpha\neq\beta$ with $t(\alpha)=t(\beta)$, then $\varepsilon(\alpha)=-\varepsilon(\beta)$.
    
    \item If $\alpha,\beta\in Q_1$ such that $s(\alpha)=t(\beta)$ and $\alpha\beta\notin\rho$, then $\varsigma(\alpha)=-\varepsilon(\beta)$.
\end{enumerate}
For each $\alpha\in Q_1$, let $\alpha^{-1}$ be the formal inverse of $\alpha$ such that ${(\alpha^{-1})}^{-1}=\alpha$. We define $s(\alpha^{-1}):=t(\alpha)$ and $t(\alpha^{-1}):=s(\alpha)$. Define $Q_1^{-1}:=\{\alpha^{-1}:\alpha\in Q_1\}$. The elements of $Q_1\sqcup Q_1^{-1}$ will be called \emph{syllables}.
\begin{defn}\label{defn: string}
A finite sequence of syllables $\alpha_k\cdots\alpha_1$ is called a \emph{(finite) string} if the following hold:
\begin{itemize}
    \item For each $i\in\{1,\cdots,k-1\}$, $t(\alpha_i)=s(\alpha_{i+1})$.

    \item For each $i\in\{1,\cdots,k-1\}$, $\alpha_i\neq\alpha_{i+1}^{-1}$.

    \item For all $1\leq i\leq i+j\leq k$, $\alpha_{i+j}\cdots\alpha_i\notin\rho$ and $\alpha_i^{-1}\cdots\alpha_{i+j}^{-1}\notin\rho$.
\end{itemize}
\end{defn}
For a string $\f x:=\alpha_k\cdots\alpha_1$, we extend the notions of inverse, starting and ending vertex in the obvious ways: $\f x^{-1}:=\alpha_1^{-1}\cdots\alpha_k^{-1}$, $s(\f x^{-1}):=t(\f x)$ and $t(\f x^{-1}):=s(\f x)$. We refer to $\alpha_1$ as the \emph{first syllable} of $\f x$ and $\alpha_k$ as the \emph{last syllable} of $\f x$. For each $v\in Q_0$, there are zero-length strings $1_{(v,1)}$ and $1_{(v,-1)}$, where we have $1_{(v,i)}^{-1}=1_{(v,-i)}$ and $s(1_{(v,i)})=t(1_{(v,i)})=v$ for any $i\in\{-1,1\}$. We denote by $\s{St}(\Lambda)$ the set of all strings in a string algebra $\Lambda$ and by $\s{St}_{>0}(\Lambda)$ the set of all positive-length strings. The $\varsigma$ and $\varepsilon$ maps are extended to $\s{St}(\Lambda)$ in the following way. For $\alpha\in Q_1$, $\varsigma(\alpha^{-1}):=\varepsilon(\alpha)$ and $\varepsilon(\alpha^{-1}):=\varsigma(\alpha)$. If $\f u=\alpha_k\cdots\alpha_1\in\s{St}_{>0}(\Lambda)$ then $\varsigma(\f u):=\varsigma(\alpha_1)$ and $\varepsilon(\f u):=\varepsilon(\alpha_k)$. In addition, we have $\varsigma(1_{(v,i)}):=-i=:-\varepsilon(1_{(v,i)})$ for each $i\in\{-1,1\}$.

For strings $\f x$ and $\f y$, if $\f y\f x$ is defined then $\varsigma(\f y)=-\varepsilon(\f x)$. For a vertex $v\in Q_0$ and $i\in\{-1,1\}$, we say that the concatenation $1_{(v,i)}\f x$ is defined and is equal to $\f x$ if $t(\f x)=v$ and $\varepsilon(\f x)=i$. Similarly, we say that the concatenation $\f x1_{(v,i)}$ is defined and is equal to $\f x$ if $s(\f x)=v$ and $\varsigma(\f x)=-i$.

To describe the nature (inverse, direct or mixed) of a positive-length string, we define a \emph{sign map} $\delta:\s{St}_{>0}(\Lambda)\to\{-1,0,1\}$ as follows. Given a string $\f u=\alpha_k\cdots\alpha_1$ define
$$\delta(\f u):=
\begin{cases}
-1&\text{ if $\alpha_i\in Q_1$ for every $i\in\{1,\cdots k\}$}\\
1&\text{ if $\alpha_i\in Q_1^{-1}$ for every $i\in\{1,\cdots k\}$}\\
0&\text{ otherwise}
\end{cases}
$$

Although a positive-length string is a finite sequence of syllables, zero-length strings sit between the syllables as is described by the following remark.

\begin{rmk}\label{rem:alt string repn}
Given a string $\f u=\f x_k\cdots\f x_1\in\mathrm{St}_{>0}(\Lambda)$, we can write $$\f u=1_{(t(\f x_k),\varepsilon(\f x_k))}\f x_k1_{(s(\f x_k),-\varsigma(\f x_k))}\cdots\f x_21_{(s(\f x_2),-\varsigma(\f x_2))}\f x_11_{(s(\f x_1),-\varsigma(\f x_1))}.$$
In particular, we will use this presentation when $\delta(\f x_i)=-\delta(\f x_{i+1})\neq0$ for $i\in\{1,\cdots,k-1\}$, in which case we say that $\f u=\f x_k\cdots\f x_1$ is the \emph{standard partition} of $\f x$.
\end{rmk}

A standard notion about finite sequences of syllables will be useful later.
\begin{defn}\label{defn: period of a string}
A \emph{period} of a string $\f x_1:=\alpha_k\cdots\alpha_1$ is a positive integer $1\leq p\leq k$ such that $\alpha_{i+p}=\alpha_i$, whenever $1\leq i+p\leq k$.    
\end{defn}

The following proposition, popularly known as the Knuth-Morris-Pratt (KMP) algorithm, will be useful in the proof of Lemma \ref{lem: nontrivial morphism between bands criterion in terms of WPCC}.

\begin{prop}\label{prop: period of a string}
\cite[Lemma~1]{knuth1977fast} Let $p_1$ and $p_2$ be periods of a finite string $\f x$. Then $\mathrm{gcd}(p_1,p_2)$ is a period of $\f x$.
\end{prop}

\begin{defn}
A string $\f b\in\s{St}_{>0}(\Lambda)$ is called a \emph{band} if all of the following hold:
\begin{itemize}
    \item $\f b$ is cyclic i.e. $s(\f b)=t(\f b)$;
    \item $\f b$ is primitive i.e. $\f b\neq\f u^m$ for any $m\geq1$ and any string $\f u$;
    \item for every $m\geq1$, $\f b^m$ is a string;
    \item the first syllable of $\f b$ is inverse and the last syllable of $\f b$ is direct.
\end{itemize}
\end{defn}
If $\f b=\alpha_k\cdots\alpha_1$ is a band, let $\sigma_i(\f b):=\alpha_{i-1}\cdots\alpha_1\alpha_k\cdots\alpha_i$ for each $i\in\{1,\cdots,k\}$. If $\delta(\alpha_i)\neq\delta(\alpha_{i-1})$ then call $\sigma_i(\f b)$ a \emph{special cyclic permutation} of $\f b$. Denote the set of special cyclic permutations in a string algebra by $\cycsp$. Also denote the set of powers of special cyclic permutations of bands by $\cyc$.

For a string $\f x$, one can associate an indecomposable module known as a \emph{string module} $M(\f x)$ (see \cite[\S~1.9]{SchHam98} or \cite[\S~2.3.1]{LakThe16}). We have $M(\f x)\cong M(\f y)$ if and only if $\f x=\f y^{-1}$. For a band $\f b$, $l\geq1$ and $\lambda\in\C K^\times$, one can associate an indecomposable module known as a \emph{band module} $B(\f b,l,\lambda)$ (see \cite[\S~1.11]{SchHam98} or \cite[\S~2.3.2]{LakThe16}). We have $B(\f b,l,\lambda)\cong B(\f b',l',\lambda')$ if and only if $l=l'$, $\lambda=\lambda'$, and $\f b'$ is a cyclic permutation of either $\f b$ or $\f b^{-1}$.

The string modules and band modules are all finite-dimensional indecomposable modules over a string algebra (see \cite{butler1987auslander}). Hence, they provide a complete description of the objects of the Krull-Schmidt category $\Lambda$-$\mathrm{mod}$. An explicit description of $\mathrm{Hom}_\Lambda(M,N)$ where $M,N$ are finite-dimensional indecomposable modules was given by Crawley-Boevey \cite{crawley1989maps} and Krause \cite{Krause1991MapsBT}. We present here such a description when $M$ and $N$ are band modules. To describe such morphisms, we need to introduce some more combinatorial objects.

Similar to finite strings, \emph{infinite strings} can be defined using an infinite sequence of syllables. A sequence $\cdots\alpha_2\alpha_1$ of syllables is called a \emph{left $\mathbb N$-string} if $\alpha_k\cdots\alpha_1$ is a string for every $k\in\mathbb N$. A sequence of syllables is called a \emph{right $\mathbb N$-string} if $\alpha_{-1}\alpha_{-2}\cdots\alpha_{-k}$ is a string for every $k\in\mathbb N$. A sequence of syllables $\cdots\alpha_1\alpha_0\alpha_{-1}\cdots$ is called a \emph{$\mathbb Z$-string} if $\alpha_i\cdots\alpha_j$ is a string for every pair of integers $i\geq j$.

For (possibly infinite) strings $\f x$ and $\f y$, we say that $\f x$ is a \emph{substring} of $\f y$, denoted $\f x\sqsubseteq\f y$, if $\f y=\f u\f x\f v$ for some (possibly infinite) strings $\f u$ and $\f v$. We say that $\f x$ is a \emph{left substring} of $\f y$, denoted $\f x\sqsubseteq_l\f y$, if $\f y=\f u\f x$ for a (possibly infinite) string $\f u$. Dually say that $\f x$ is a \emph{right substring} of $\f y$, denoted $\f x \sqsubseteq_r \f y$, if $\f y=\f x\f u$ for a (possibly infinite) string $\f u$. We say that $\f x$ is a proper (resp. left, right) substring of $\f y$ if $\f x\sqsubseteq \f y$ (resp. $\f x\sqsubseteq_l \f y$, $\f x\sqsubseteq_r \f y$) but $\f x\neq\f y$.

\begin{defn}\label{defn:factor image}
Given a (possibly infinite) string $\f x$, a (possibly infinite) substring $\f u$ of $\f x$ is said to be a \emph{factor substring} if one of the following holds:
\begin{itemize}
    \item $\f u=\f x$;
    \item $\alpha\f u\sqsubseteq_l\f x$ for some $\alpha\in Q_1$;
    \item  $\f u\beta^{-1}\sqsubseteq_r\f x$ for some $\beta\in Q_1$;
    \item $\alpha\f u\beta^{-1}\sqsubseteq\f x$ for some $\alpha,\beta\in Q_1$.
\end{itemize}
Dually, a (possibly infinite) substring $\f u$ of a (possibly infinite) string $\f x$ is said to be an \emph{image substring} if one of the following holds:
\begin{itemize}
    \item $\f u=\f x$;
    \item $\alpha^{-1}\f u\sqsubseteq_l\f x$ for some $\alpha\in Q_1$;
    \item $\f u\beta\sqsubseteq_r\f x$ for some $\beta\in Q_1$;
    \item $\alpha^{-1}\f u\beta\sqsubseteq\f x$ for some $\alpha,\beta\in Q_1$.
\end{itemize}
\end{defn}

It will be useful to mark the location of a particular substring of a (possibly infinite) string instead of just mentioning the substring. For example, the string $aB$ is a substring of the string $aBaB$ but it occurs multiple times, so it is unclear as to which substring we are referring to.

\begin{defn}\label{defn: marked string}
Given a (possibly infinite) string $\f x$, define \emph{the indexing set} $I(\f x)$ of $\f x$ as follows.
\begin{itemize}
    \item If $\f x$ is finite then $I(\f x):=\{0,1,2,\cdots,|\f x|\}$.
    \item If $\f x$ is a left $\mathbb N$-string then $I(\f x):=\mathbb Z^+\cup\{0,\infty\}$.
    \item If $\f x$ is a right $\mathbb N$-string then $I(\f x):=\mathbb Z^-\cup\{-\infty\}$.
    \item If $\f x$ is a $\mathbb Z$-string then $I(\f x):=\mathbb Z\cup\{-\infty,\infty\}$.
\end{itemize}

A \emph{marked string} is a triple $\dot{\f x}:=(\f x,k_1,k_2)$, where $\f x$ is a (possibly infinite) string, $k_1,k_2\in I(\f x)$ with $k_1\leq k_2$. Denote the set of all marked strings by $\marst$.
\end{defn}

\begin{exmp}
Consider the string $\f x:=ab^{-1}ab^{-1}ab^{-1}$ in $\Lambda_2$. Then the marked string $\dot{\f x_1}:=(\f x,0,2)$ denotes $ab^{-1}ab^{-1}\underline{ab^{-1}}$, whereas $\dot{\f x_2}:=(\f x,2,4)$ denotes $ab^{-1}\underline{ab^{-1}}ab^{-1}$.
\end{exmp}

Let $\dot{\f x}:=(\f x,k_1,k_2)$ be a marked string. Define maps $\vartheta_l,\vartheta_r:\marst\to\{-1,0,1\}$ by $$\vartheta_l(\dot{\f x}):=
\begin{cases}
    1&\text{ if }\delta(\alpha_{k_2+1})=1,\\
    -1&\text{ if }\delta(\alpha_{k_2+1})=-1,\\
    0&\text{ if }\alpha_{k_2+1}\text{ does not exist},
\end{cases}
\hspace{2mm} 
\vartheta_r(\dot{\f x}):=
\begin{cases}
    1&\text{ if }\delta(\alpha_{k_1})=-1,\\
    -1&\text{ if }\delta(\alpha_{k_1})=1,\\
    0&\text{ if }\alpha_{k_1}\text{ does not exist}.    
\end{cases}
$$

\begin{rmk}\label{rem: factor image by marked string}
Let $\dot{\f x}:=(\f x,k_1,k_2)$ be a marked string. Then $\f x$ can be written as $\f x_2\f u\f x_1$, where the last syllable of $\f x_1$ is $\alpha_{k_1}$ if $|\f x_1|>0$ and the first syllable of $\f x_2$ is $\alpha_{k_2+1}$ if $|\f x_2|>0$. If $\vartheta_l(\dot{\f x}),\vartheta_r(\dot{\f x})\in\{0,-1\}$ (resp. $\{0,1\}$) then $\f u$ is a factor (resp. image) substring of $\f x$.
\end{rmk}

For a string $\f x$, let $\INF{\f x}$ denote the $\mathbb Z$-string obtained by concatenation of countably many copies of $\f x$ on left as well as right. There is an action of $\mathbb Z$ on the set of finite substrings of $\INF{\f x}$ by translation. Let $\FacZ{\f x}$ and $\ImZ{\f x}$ denote the sets of finite factor substrings of $\INF{\f x}$ and finite image substrings of $\INF{\f x}$ respectively up to this $\mathbb Z$-action.

For band modules $M$ and $N$, Krause \cite{Krause1991MapsBT} gave a basis of $\mathrm{Hom}_\Lambda(M,N)$--the elements of this basis are known as \emph{graph maps}. This is the content of the following theorem whose presentation closely follows \cite[Theorem~3.12]{dequêne2023generalization}. 
\begin{thm}
Let $\f b$ and $\f b'$ be two bands, $l,l'\in\mathbb N$, and $\lambda,\lambda'\in\C K^\times$.
\begin{enumerate}
\item \label{morphisms between bands first condition} If $B(\f b,1,\lambda)\not\cong B(\f b',1,\lambda')$, then a basis of $\mathrm{Hom}_\Lambda(B(\f b,l,\lambda),B(\f b',l',\lambda'))$ is in bijection with the set 
\begin{align*}
\{(\f x,\f y)\in\FacZ{\f b}\times\ImZ{\f b'}\mid\f x=\f y\text{ or }\f x=\f y^{-1}\}\times V_{l,l'},   
\end{align*}
where $V_{l,l'}$ is a basis of the vector space $\mathrm{Hom}_\C K(\C K^l,\C K^{l'})$.

\item If $B(\f b,1,\lambda)\cong B(\f b',1,\lambda')$, then a basis of $\mathrm{Hom}_\Lambda(B(\f b,l,\lambda),B(\f b',l',\lambda'))$ is in bijection with the set mentioned in (\ref{morphisms between bands first condition}) union with a basis of the space $\mathrm{Hom}_{\C K[x]}\left(\C K[x]/(x^l),\C K[x]/(x^{l'})\right)$.
\end{enumerate}
\end{thm}

The following two consequences of the above theorem will be crucial in \S~\ref{sec: main} where we give a combinatorial characterisation of band bricks.

\begin{cor}\label{cor:existence of morphisms between band modules}
Let $\f b_1$ and $\f b_2$ be bands, $\lambda_1,\lambda_2\in\C K^\times$. There exists a non-trivial(=non-zero and non-identity) morphism $B(\f b_1,1,\lambda_1)\to B(\f b_2,1,\lambda_2)$ if and only if there is a finite string $\f u$ that occurs as a factor substring of $\INF{\f b_1}$ as well as an image substring of $\INF{\f b_2}$ or $\INF{(\f b_2^{-1})}$.
\end{cor}

The next corollary gives a criterion for a band module to be a brick.
\begin{cor}\label{cor: brick criteria}
Let $\f b$ be a band, $l\in\mathbb N$ and $\lambda\in\C K^\times$. Then the band module $B(\f b,l,\lambda)$ is a brick if and only if $l=1$ and there is no finite string $\f x$ that is a factor substring of $\INF{\f b}$ as well as an image substring of either $\INF{\f b}$ or $\INF{(\f b^{-1})}$.
\end{cor}

\section{Band Bricks for $\Lambda_N$}\label{sec: band bricks in lamdaN}
In this section, we give an overview of the result proved in \cite{dequêne2023generalization} along with an extension to that result. To begin with, we need some notations and terminologies related to combinatorics of words.

A non-empty finite linearly ordered alphabet will be denoted $\s A$ unless otherwise stated. The collection of all finite-length $\s A$-words (including the empty word) will be denoted $\s A^*$. Its subset consisting of the positive-length words will be denoted $\s A^+$. Unlike strings, we read words from left to right. For example, if $\s w:=n_1n_2\cdots n_k$ is a word then $n_1$ is the first letter of $\s w$ and $n_k$ is the last letter of $\s w$. 

\begin{defn}\label{defn: pcw}
\cite[\S~1.1]{dequêne2023generalization} Let $\s w\in\s A^+$. The \emph{Burrows-Wheeler transformation} of $\s w$, denoted by $\s{BW}(\s w)$, is defined as follows. Let $\s w_1,\cdots,\s w_k$ be all the cyclic permutations of $\s w$ arranged in the lexicographic order. Let $n_1,\cdots,n_k$ be the last letters of $\s w_1,\cdots,\s w_k$ respectively. Then $\s{BW}(\s w):=n_1n_2\cdots n_k$. If $\s w$ is primitive and we have $n_1\geq n_2\geq \cdots\geq n_k$ then we say that $\s w$ is \emph{perfectly clustering}.
\end{defn}

A \emph{primitive $\s A$-word} $\s w$ is a word that is not a $k$-fold concatenation of a word for $k\geq2$. An equivalent criterion for a primitive $\s A$-word to be perfectly clustering is as follows.

\begin{prop}\label{prop: eq criteria pcw}
\cite[Proposition~5.4]{dequêne2023generalization} A primitive $\s A$-word $\s w$ is perfectly clustering if and only if there are no cyclic permutations $\s w'=n'\s zm'\s z'$ and $\s w''=n''\s zm''\s z''$ of $\s w$ for some $\s z,\s z',\s z''\in\s A^*$ such that $n'<n''$ and $m'<m''$.
\end{prop}

\begin{exmp}
Let $\s A:=\{1,2,3,4\}$ with the usual ordering be a linearly ordered alphabet. The $\s A$-word $121434$ is not a perfectly clustering word as two of its cyclic permutations are $121434$ and $341214$, and $1<3$ and $2<4$.
\end{exmp}

For the convenience of the reader, we repeat in Figure \ref{fig:lambda_n} the presentation of the gentle algebra $\Lambda_N$, for $N\geq2$, studied in \cite{dequêne2023generalization}.
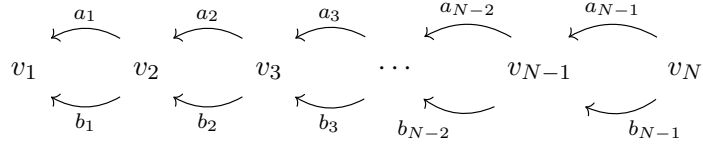
\begin{figure}[h]
    \centering
    \[\begin{tikzcd}
v_1 & v_2 \arrow[l, "b_1", bend left, shift left=2] \arrow[l, "a_1"', bend right, shift right=2] & v_3 \arrow[l, "b_2", bend left, shift left=2] \arrow[l, "a_2"', bend right, shift right=2] & \cdots \arrow[l, "a_3"', bend right, shift right=2] \arrow[l, "b_3", bend left, shift left=2] & v_{N-1} \arrow[l, "a_{N-2}"', bend right, shift right=2] \arrow[l, "b_{N-2}", bend left, shift left=2] & v_N \arrow[l, "b_{N-1}", bend left, shift left=2] \arrow[l, "a_{N-1}"', bend right, shift right=2]
\end{tikzcd}\]
    \caption{$\Lambda_N$ with $\rho=\{a_1b_2,a_2b_3,\cdots,a_{N-2}b_{N-1},b_1a_2,b_2a_3,\cdots,b_{N-2}a_{N-1}\}$}
    \label{fig:lambda_n}
\end{figure}
Let $\s A:=\{2,\cdots,N\}$ be an alphabet linearly ordered with the usual ordering. Associate to each integer $n\in\s A$ a cyclic string $\f b_n:=a_1a_2\cdots a_{n-1}b_{n-1}^{-1}\cdots b_2^{-1}b_1^{-1}$. Given a primitive word $\s z=n_1\cdots n_k$, define $\varphi(\s z):=\f b_{n_k}\cdots\f b_{n_1}$. This defines an injective map $\varphi$ from the set of primitive words to the set of bands in $\Lambda_N$.

The following result is one of the main results of \cite{dequêne2023generalization} which characterises whether a word $\s w$ is perfectly clustering in terms of whether the band module $B(\varphi(\s w),1,\lambda)$ is a brick. 

\theirmainthm*

Note that the map $\varphi$ is not surjective. For example, the band $a_1a_2b_2^{-1}a_2b_2^{-1}b_1^{-1}$ does not appear in the image of $\varphi$. In order to describe an appropriate extension $\widetilde{\varphi}$ of $\varphi$, we extend the alphabet $\s A$ to $\widetilde{\s A}:=\{1,2,\cdots,N\}$ with the usual ordering. The $\widetilde{\s A}$-words corresponding to elements of $\s{Cyc}(\Lambda_N)$ will be some special words called \emph{crowns}.

\begin{defn}\label{defn: zigzagcrown}
A word $\s w=n_1\cdots n_k$ is called \emph{zigzag} if for every $i\in\{1,\cdots,k-1\}$ exactly one of the following happens:
\begin{itemize}
    \item $n_i<n_{i+1}$ if and only if $i$ is odd;
    \item $n_i<n_{i+1}$ if and only if $i$ is even.
\end{itemize}
Denote the set of all zigzags over an alphabet $\s A$ by $\ZZ$.

A zigzag $\s w$ is called a \emph{crown} if every cyclic permutation of $\s w$ is a zigzag. Denote the set of all crowns over an alphabet $\s A$ by $\crowns$.
\end{defn}

\begin{rmk}
It is trivial to note that every crown is of even length.
\end{rmk}

The standard partition of a string described in Remark \ref{rem:alt string repn} provides a way to describe the inverse of the yet-to-be-defined extension $\widetilde\varphi$ of $\varphi$ by noting down the vertices where signs change. Given a band $\f b$, let its standard partition be $\f b=\f x_{2k}\cdots\f x_2\f x_1$. The word $\widetilde{\varphi}^{-1}(\f b)$ is defined to be $v_{n_1}\cdots v_{n_{2k}}$, where $v_{n_i}=s(\f x_i)$ for $i\in\{1,\cdots,2k\}$. Noting that if $\delta(\f x_i)=-1$ (resp. $\delta(\f x_i)=1$) and $t(\f x_i)=v_j$ then $j>i$ (resp. $j<i$), we conclude that $\widetilde{\varphi}^{-1}(\f b)$ is a crown.

\begin{exmp}\label{exmp: new association in lambdaN}
The standard partition of $\f b:=a_1a_2b_2^{-1}a_2b_2^{-1}b_1^{-1}$ is $(a_1a_2)(b_2^{-1})(a_2)(b_2^{-1}b_1^{-1})$, and hence $\widetilde\varphi^{-1}(\f b)=1324$.
\end{exmp}

Before we can define $\widetilde\varphi$, we need the following fact about string algebras.

\begin{prop}\label{prop:string at most one sequence}
Let $\Lambda$ be a string algebra whose underlying quiver is acyclic. Given $\f x, \f y\in\s{St}(\Lambda)$, there is at most one $\f z\in\s{St}(\Lambda)$ with $\delta(\f z)\neq0$ such that $\f y\f z\f x\in\s{St}(\Lambda)$.
\end{prop}
\begin{proof}
If possible, let $\f z_1,\f z_2$ be two distinct strings with $\delta(\f z_1),\delta(\f z_2)\neq 0$ such that $\f y\f z_i\f x\in\s{St}(\Lambda)$ for each $i=1,2$.

\textbf{Case 1.} $\delta(\f z_1)=\delta(\f z_2)$.

Let $\f w$ be the maximal common left substring of $\f z_1$ and $\f z_2$. If $\f w\sqsubset_l\f z_1$ and $\f w\sqsubset_l\f z_2$ then there exist distinct syllables $\alpha$ and $\beta$ with $\delta(\alpha)=\delta(\beta)$ such that $\alpha\f w$ and $\beta\f w$ are strings, which violates the definition of a string algebra (Condition \ref{at most one arrow} of Definition \ref{defn:string alg}).

Therefore, without loss, assume that $\f w=\f z_1$, which gives $\f y\f z_2\f x=\f y\f z'\f z_1\f x$ for some string $\f z'$ with $|\f z'|>0$ and $\delta(\f z')=\delta(\f z_2)$. Since $\f y\f z'\f z_1$ and $\f y\f z_1$ are strings we have $s(\f z')=t(\f z')$, which, given $\delta(\f z')\neq 0$, implies the existence of a directed cycle in the underlying quiver of $\Lambda$, a contradiction.
\\

\textbf{Case 2.} $\delta(\f z_1)\neq\delta(\f z_2)$.

Without loss, let $\f z_1=\alpha_k\cdots\alpha_1$ and $\f z_2=\beta_l^{-1}\cdots\beta_1^{-1}$, where each $\alpha_i,\beta_j\in Q_1$ for $i\in\{1,\cdots,k\}$ and $j\in\{1,\cdots,l\}$. Then we have a directed cycle $$t(\f z_1)\xrightarrow{\alpha_1}\cdots\xrightarrow{\alpha_k}s(\f z_1)\xrightarrow{\beta_l}\cdots\xrightarrow{\beta_1}t(\f z_1)$$ in the underlying quiver of $\Lambda$, a contradiction.
\end{proof}

Since the underlying quiver of $\Lambda_N$ is acyclic, the above proposition is applicable. Moreover, for $i,j\in\{1,2\cdots,N\}$, the existence of a string $\f z$ with $\delta(\f z)\neq0$ such that $1_{(v_i,1)}\f z1_{(v_j,1)}$ is a string is also guaranteed. Choose $\varsigma$ and $\varepsilon$ maps such that $1_{(v_{n+1},1)}b_n^{-1}1_{(v_n,1)}$ is a string for every $n\in\{1,\cdots N-1\}$. To summarize, given $i<j$, the only string $\f x$ such that $1_{(v_j,1)}\f x1_{(v_i,1)}$ is a string is $b_{j-1}^{-1}\cdots b_{i+1}^{-1}b_i^{-1}$, and the only string $\f x$ such that $1_{(v_i,1)}\f x1_{(v_j,1)}$ is a string is $a_i\cdots a_{j-2}a_{j-1}$.

Now we are ready to define $\widetilde\varphi:\s{Crowns}(\widetilde{\s A})\to\s{Cyc}(\Lambda_N)$. Given an $\widetilde{\s A}$-crown $\s w:=n_1n_2\cdots n_k$, set $\widetilde{\varphi}(\s w):=\f x_k\f x_{k-1}\cdots\f x_2\f x_1$, where $\f x_i$ is the unique string such that $1_{(v_{n_{i+1}},1)}\f x_i1_{(v_{n_i},1)}$ is a string for $i\in\{1,\cdots,k-1\}$, and $\f x_k$ is the unique string such that $1_{(v_{n_1},1)}\f x_i1_{(v_k,1)}$ is a string. The restriction of this map to the set of primitive crowns gives a one-to-one correspondence between the set of primitive crowns and the set of special cyclic permutations of bands. 

\begin{exmp}
Continuing from Example \ref{exmp: new association in lambdaN}, we indeed get $\widetilde\varphi(1324)=a_1a_2a_3b^{-1}_3b^{-1}_2a_2b^{-1}_2b^{-1}_1.$
\end{exmp}

The following result states that $\widetilde{\varphi}^{-1}\circ\varphi$ preserves and reflects perfectly clustering $\widetilde{\s A}$-words.

\begin{prop}\label{prop: new association consistent}
Suppose $\s A:=\{2,\cdots,N\}\subset\widetilde{\s A}:=\{1,\cdots,N\}$ are linearly ordered alphabets with the usual ordering. Let $\s w:=n_1\cdots n_k$ be an $\s A$-word and $\widetilde{\s w}:=\widetilde{\varphi}^{-1}(\varphi(\s w))=1n_11n_2\cdots n_{k-1}1n_k$ be the corresponding $\widetilde{\s A}$-word. Then $\s w$ is perfectly clustering if and only if $\widetilde{\s w}$ is perfectly clustering.
\end{prop}
\begin{proof}
Suppose $\s w$ is not perfectly clustering. Then by Proposition \ref{prop: eq criteria pcw}, there are cyclic permutations $\s w':=n'\s zm'\s z'$ and $\s w'':=n''\s zm''\s z''$ of $\s w$ such that $n'<n''$ and $m'<m''$. Let $\s z:=m_1\cdots m_l$ and $\widetilde{\s z}:=1m_11\cdots 1m_l$. Then we have cyclic permutations $n'\widetilde{\s z}1m'\widetilde{\s z}'$ and $n''\widetilde{\s z}1m''\widetilde{\s z}''$ of $\widetilde{\s w}$ for some $\widetilde{\s A}$-words $\widetilde{\s z}'$ and $\widetilde{\s z}''$ with $n'<n''$ and $m'<m''$, which implies, by Proposition \ref{prop: eq criteria pcw}, that $\widetilde{\s w}$ is not perfectly clustering.

Conversely, suppose $\widetilde{\s w}$ is not perfectly clustering. Again by Proposition \ref{prop: eq criteria pcw}, there are cyclic permutations $\widetilde{\s w}':=n'\widetilde{\s z}m'\widetilde{\s z}'$ and $\widetilde{\s w}'':=n''\widetilde{\s z}m''\widetilde{\s z}''$ of $\widetilde{\s w}$ such that $n'<n''$ and $m'<m''$. If $n'=1$ then the first letter of $\widetilde{\s z}$ is not 1 and therefore $n''=1$, a contradiction. So $n'\neq 1$. Similarly, $m'\neq 1$. Therefore, $\widetilde{\s z}=1\widetilde{m}_11\cdots1\widetilde{m}_l1$ for some $l\geq0$ and $\widetilde{m}_i\in\s A$ for each $i\in\{1,\cdots,l\}$. Consequently, we have that $n'\widetilde{m}_1\cdots\widetilde{m}_lm'\s z'$ and $n''\widetilde{m}_1\cdots\widetilde{m}_lm''\s z''$ are cyclic permutations of $\s w$ with $n'<n''$ and $n''<m''$, where $\s z'$ and $\s z''$ are obtained by removing all occurrences of $1$ in $\widetilde{\s z}'$ and $\widetilde{\s z}''$ respectively. Therefore, once again, by Proposition \ref{prop: eq criteria pcw}, we conclude that $\s w$ is not perfectly clustering.
\end{proof}

We restrict the criterion (Proposition \ref{prop: eq criteria pcw}) for words to be perfectly clustering to define \emph{weakly perfectly clustering words}.

\begin{defn}\label{defn: weak perf clus}
A primitive crown $\s w$ is called \emph{weakly perfectly clustering} if none of the following happens.
\begin{enumerate}
    \item \label{axiom1 of WPCC} There exist cyclic permutations $\s w'$ and $\s w''$ of $\s w$ such that $\s w'=n'\s zm'\s z'$ and $\s w''=n''\s zm''\s z''$, where $n'<n'',\ m<m''$ and $|\s z|\geq 1$.
    \item \label{axiom2 of WPCC}  There exist cyclic permutations $\s w'$ and $\s w''$ of $\s w$ such that $\s w'=n'm'\s z'$ and $\s w''=n''m''\s z''$, where $n'<n'',\ m'<m''$, $n'<m'$ if and only if $n''<m''$, and $[n',m']\cap[n'',m'']\neq\emptyset$.
\end{enumerate}
\end{defn}
The following result which can be obtained as a corollary of Theorem \ref{thm: main} extends Theorem \ref{thm: their main thm}; the proof is deferred until \S~\ref{sec: main}.

\brickiffwpcc*

\section{Finite Traced Poset}\label{sec: fin traced poset}
So far the alphabet in consideration was linearly ordered. The ordering was based on the rule that if there is a path from $v_i$ to $v_j$ then $v_i>v_j$. However, to deal with more general string algebras, note that we need to consider alphabets that are not necessarily linearly ordered (it is trivial to construct such an example). Even this is not sufficient as the following example suggests.

\begin{exmp}\label{exmp: motivation of traced poset}
Consider the string algebra $\Gamma$ in Figure \ref{fig:motivation of traced poset} and the band $\f b=becd^{-1}ea^{-1}$. It can be checked that the band module $B(\f b,1,\lambda)$ for any $\lambda\in\C K^\times$ is a brick.

As discussed above, the alphabet $\s A:=\{v_1,v_2,v_3,v_4\}$ can be linearly ordered as follows:
$$v_1<v_2<v_3<v_4$$
Recall the definition of $\widetilde{\varphi}$ from the previous section. We have $\widetilde{\varphi}^{-1}(\f b)=v_1v_3v_2v_4$. But note that $v_1v_3v_2v_4$ is not a weakly perfectly clustering word over the linearly ordered alphabet $\s A$ as $v_1v_3v_2v_4$ and $v_2v_4v_1v_3$ are cyclic permutations of $\f b$ with $v_1<v_2$, $v_3<v_4$ and $[v_1,v_3]\cap[v_2,v_4]=\{v_2,v_3\}\neq\emptyset$. However, the paths $a$ and $d$ witnessing $v_1v_3$ and $v_2v_4$ respectively are disjoint.
\begin{figure}[h]
    \centering
    \begin{tikzcd}
    & v_3 \arrow[ld, "a"'] \arrow[dd, "e"] &                                      \\
v_1 &                                      & v_4 \arrow[ld, "d"] \arrow[lu, "c"'] \\
    & v_2 \arrow[lu, "b"]                  &                                     
\end{tikzcd}
    \caption{$\Gamma$ with $\rho=\{ac,bd\}$}
    \label{fig:motivation of traced poset}
\end{figure}
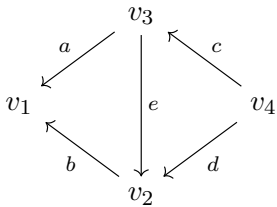
\end{exmp}

The above example motivates us to ``keep track of/trace'' appropriate paths in a partial order obtained from a string algebra. Thus we define a \emph{finite traced partially ordered set} below. We will see later that this serves as an alphabet to form crowns. Before that, we need a preliminary definition from order theory that is used to draw Hasse diagrams corresponding to certain posets.

\begin{defn}
Let $(P,<)$ be a strict partially ordered set and $x<y$ be elements of $P$. The element $y$ is said to \emph{cover} $x$ (equivalently, $x$ is \emph{covered by} $y$) if there is no element $z\in P$ such that $x<z<y$. 
\end{defn}

\begin{defn}\label{defn: traced poset}
A finite \emph{traced partially ordered set} $\s A=(A,<,T,\mu:A\to A)$ is a finite strict partially ordered set $(A,<)$ along with a set $T$ of strict monotone sequences of elements from $A$ that is closed under (possibly empty) subsequences satisfying the following properties.
\begin{enumerate}
    \item \label{at most two elements cover} For each $n\in A$, there are at most two elements which cover $n$. 
    \item [(1$'$)] \label{at most two elements cover op} For each $n\in A$, there are at most two elements covered by $n$.
    
    \item \label{at most one sequence} For each pair of elements $n,m\in A$, there is at most one sequence $\langle x_k\rangle^N_{k=1}$ in $T$ such that $x_1=n$ and $x_N=m$.
\end{enumerate}
In view of the above condition, let $$T_{(n,m)}:=
\begin{cases}
\langle x_k\rangle_{k=1}^N  &\text{ if such a sequence exists;}\\
\langle\rangle &\text{ otherwise.}
\end{cases}$$
\begin{enumerate}[start=3]
   \item \label{successor is a cover} Given $n<m$ in $A$, if $T_{(n,m)}=\langle x_k\rangle_{k=1}^N$ is a non-empty sequence then $x_k$ is a cover of $x_{k-1}$ for every $k\in\{2,\cdots,N\}$.

    \item[(3$'$)] Given $n>m$ in $A$, if $T_{(n,m)}=\langle x_k\rangle_{k=1}^N$ is a non-empty sequence then $x_{k-1}$ is a cover of $x_k$ for every $k\in\{2,\cdots,N\}$.
    
    \item\label{at most one essential neighbour} For each $m\in A$, there is at most one $n\in A$ such that $m<n$ and $T_{(m,n)}=\langle m,n\rangle$.

     \item[(4$'$)] For each $m\in A$, there is at most one $n\in A$ such that $n<m$ and $T_{(n,m)}=\langle n,m\rangle$.

     \item \label{mu map} The map $\mu:(A,<)\to(A,<)$ is an involution, isomorphism, and has no fixed points.

     \item \label{T and mu} If $T_{(n,m)}=\langle x_k\rangle_{k=1}^N$ then $T_{(\mu(m),\mu(n))}=\langle\mu(x_{N-k+1})\rangle_{k=1}^N$.



\end{enumerate}
\end{defn}

\begin{rmk}
Note that if $\s A$ is a finite traced partially ordered set then $\s A^{op}:=(A,<^{op},T^{op},\mu)$ is a finite traced partially ordered set.
\end{rmk}

The examples of finite traced posets are deferred until the next section.

The definition of zigzags and crowns naturally extends to partially ordered alphabets. However, to keep track of the data of traces, we define \emph{valid zigzags} and \emph{valid crowns} over a finite traced partially ordered set.

\begin{defn}\label{defn: valid zigzag}
An $\s A$-zigzag $\s w=n_1\cdots n_k$ is called a \emph{valid zigzag} if $T_{(n_i,n_{i+1})}$ is a non-empty sequence for each $i\in\{1,\cdots,k-1\}$. Denote the set of valid zigzags by $\ZZval$.
\end{defn}

\begin{defn}\label{defn: valid crown2}
An $\s A$-crown $\s w=n_1\cdots n_k$ is called \emph{valid} if $n_1\cdots n_kn_1$ is a valid zigzag. Denote the set of all valid $\s A$-crowns by $\crownsval$. Also denote the set of all primitive valid $\s A$-crowns by $\pcrownsval$.
\end{defn}

\begin{defn}\label{defn: spec crown}
A valid $\s A$-crown $\s w=n_1\cdots n_k$ is called \emph{special} if $n_1<n_2$.
Denote the set of all special $\s A$-crowns by $\crownsspec$. Also denote the set of all primitive special $\s A$-crowns by $\pcrownsspec$.
\end{defn}

The definition of weakly perfectly clustering words over a linearly ordered alphabet extends naturally to weakly perfectly clustering valid words over a finite traced poset. However, we need to define its further generalization, namely the concept of a \emph{weakly perfectly clustering pair of words} over a finite traced poset, that will be used later to give a criterion (Lemma \ref{lem: nontrivial morphism between bands criterion in terms of WPCC}) for the existence of a non-trivial morphism between certain band modules.

\begin{defn}\label{defn: weak perf clus2}
Given crowns $\s w_1$ and $\s w_2$, we say that the pair $(\s w_1,\s w_2)$ is \emph{weakly perfectly clustering} if none of the following happens.
\begin{enumerate}
    \item \label{axiom1 of WPCC2} There exist cyclic permutations $\s w'$ of $\s w_1^{2\cdot|\s w_2|}$ and $\s w''$ of $\s w_2^{2\cdot|\s w_1|}$ such that $\s w'=n'\s zm'\s z'$ and $\s w''=n''\s zm''\s z''$, where $n'<n'',\ m'<m''$ and $|\s z|\geq 1$.
    \item \label{axiom2 of WPCC2}  There exist cyclic permutations $\s w'$ of $\s w_1^{2\cdot|\s w_2|}$ and $\s w''$ of $\s w_2^{2\cdot|\s w_1|}$ such that $\s w'=n'm'\s z'$ and $\s w''=n''m''\s z''$, where $n'<n'',\ m'<m''$, $n'<m'$ if and only if $n''<m''$, and $T_{(n',m')}\cap T_{(n'',m'')}\neq\emptyset$.
\end{enumerate}
\end{defn}
\begin{rmk}\label{rmk: w is WPC iff ww is WPC}
Note that a valid crown $\s w$ is weakly perfectly clustering if and only if the pair $(\s w,\s w)$ is weakly perfectly clustering.
\end{rmk}

\section{From a string algebra to a finite traced poset}\label{sec:obtain traced poset}
The examples of string algebras that we saw so far were simple in the sense that there were no instances of a string $\f x$ such that $1_{(v,{-i})}\f x1_{(v,i)}$ is a string for a vertex $v$ and $i\in\{-1,1\}$. In other words, given a band $\f b$, if $1_{(v,i)}$ is a factor (resp. image) substring of $\INF{\f b}$ then $1_{(v,-i)}$ can never be image (resp. factor) substring of either $\INF{\f b}$ or $\INF{(\f b^{-1})}$. This is the reason why it was sufficient to consider only vertices in case of $\Lambda_N$. But the following example suggests that this is not true in case of more general string algebras.

\begin{exmp}\label{exmp: motivation of cov quiv}
Consider the string algebra $\Gamma'$ from Figure \ref{fig:motivation of cov quiv} and the band $\f b:=b_3b_2^{-1}b_1a_1^{-1}a_2a_3^{-1}$. Note that $\f b$ can be written as $b_3b_2^{-1}b_11_{(v_1,-i)}a_1^{-1}a_2a_3^{-1}1_{(v_1,i)}$ from which it is clear that $1_{(v_1,-i)}$ is a factor substring of $\INF{\f b}$ and $1_{(v_1,i)}$ is an image substring of $\INF{\f b}$.
\end{exmp}

\begin{figure}
    \centering
    \begin{tikzcd}
v_3 \arrow[rd, "a_3"] \arrow[dd, "a_2"'] &                                         & v_4                                      \\
                                         & v_1 \arrow[ru, "b_1"] \arrow[ld, "a_1"] &                                          \\
v_2                                      &                                         & v_5 \arrow[uu, "b_2"'] \arrow[lu, "b_3"]
\end{tikzcd}
    \caption{$\Gamma'$ with $\rho=\{b_1a_3,a_3b_1\}$}
    \label{fig:motivation of cov quiv}
\end{figure}
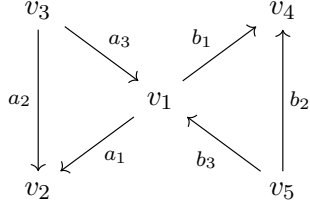

The above example motivates us to distinguish between the two zero-length strings corresponding to a vertex as elements of the finite traced poset. Therefore, we construct a quiver which we call \emph{the covering quiver} which ``unravels'' the underlying quiver of the given string algebra. As mentioned in the introduction, this covering quiver gives us a sense of the direction of strings in a string algebra.

\begin{defn}\label{defn: cov quiv}
Let $\Lambda$ be a string algebra whose underlying quiver is $Q=(Q_0,Q_1,s,t)$. Fix a choice of maps $\varsigma$ and $\varepsilon$. The \emph{covering quiver} of $\Lambda$, $\cq{}=(\cq{0},\cq{1},\cs,\ct)$, is defined as follows:
$$\cq{0}:=Q_0\times\{-1,1\},\ \cq{1}:=Q_1\cup Q_1^{-1},\ \cs(\alpha):=(s(\alpha),-\varsigma(\alpha)),\ \ct(\alpha):=(t(\alpha),\varepsilon(\alpha)).$$
\end{defn}

\begin{rmk}
We have a natural map $1_{(-)}:\cq{0}\to\s{St}(\Lambda)$ which takes $(v,i)$ to $1_{(v,i)}$.
\end{rmk}

Let us recall from \cite[\S~1]{mousavand2022tau} the class of (generalised) barbell algebras. We present here (Figure \ref{fig: barbell}) a particular generalised barbell algebra and use it as a running example.

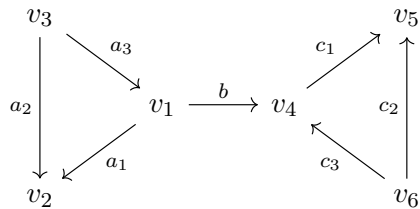
\begin{figure}[h]
    \centering
    \begin{tikzcd}
v_3 \arrow[rd, "a_3"] \arrow[dd, "a_2"'] &                                      &                       & v_5                                     \\
                                         & v_1 \arrow[r, "b"] \arrow[ld, "a_1"] & v_4 \arrow[ru, "c_1"] &                                         \\
v_2                                      &                                      &                       & v_6 \arrow[lu, "c_3"] \arrow[uu, "c_2"]
\end{tikzcd}
    \caption{$\Gamma''$ with $\rho=\{c_1ba_3,a_1a_3,c_1c_3\}$}
    \label{fig: barbell}
\end{figure}

\begin{exmp}\label{exmp: covering quiver}
Consider the string algebra $\Gamma''$ in Figure \ref{fig: barbell}. We fix $\varsigma$ and $\varepsilon$ in the following way:
\begin{align*}
(\varsigma(a_1),\varepsilon(a_1))=(1,-1),\,&(\varsigma(a_2),\varepsilon(a_2))=(1,1),\,(\varsigma(a_3),\varepsilon(a_3))=(-1,1),\\
&(\varsigma(b),\varepsilon(b))=(-1,1),\\
(\varsigma(c_1),\varepsilon(c_1))=(-1,1),\,&(\varsigma(c_2),\varepsilon(c_2))=(1,-1),\,(\varsigma(c_3),\varepsilon(c_3))=(-1,-1)
\end{align*}
Figure \ref{fig:covering quiver} shows the covering quiver for $\Gamma''$.

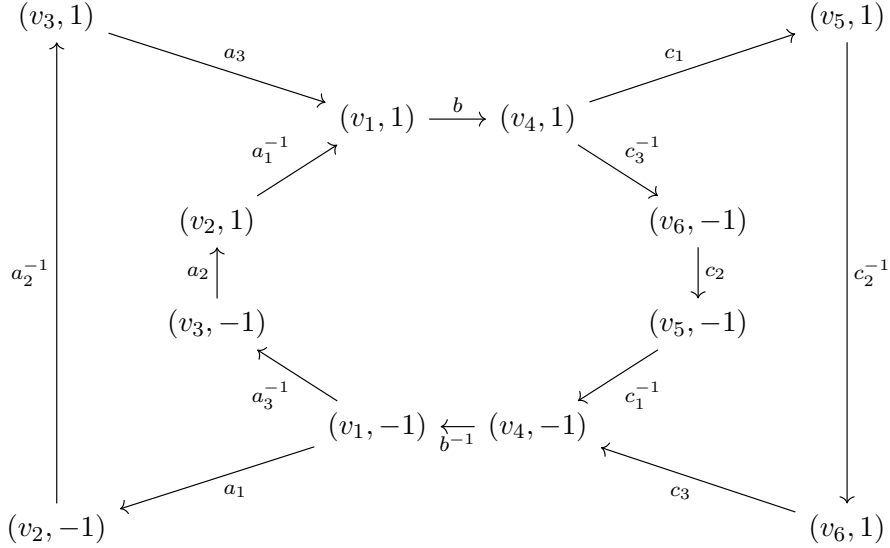
\begin{figure}[h]
    \centering
    \begin{tikzcd}[column sep=small]
{(v_3,1)} \arrow[rrd, "a_3"]         &                                  &                                                      &                                                     &                                   & {(v_5,1)} \arrow[ddddd, "c_2^{-1}"] \\
                                     &                                  & {(v_1,1)} \arrow[r, "b"]                             & {(v_4,1)} \arrow[rd, "c_3^{-1}"] \arrow[rru, "c_1"] &                                   &                                     \\
                                     & {(v_2,1)} \arrow[ru, "a_1^{-1}"] &                                                      &                                                     & {(v_6,-1)} \arrow[d, "c_2"]       &                                     \\
                                     & {(v_3,-1)} \arrow[u, "a_2"]      &                                                      &                                                     & {(v_5,-1)} \arrow[ld, "c_1^{-1}"] &                                     \\
                                     &                                  & {(v_1,-1)} \arrow[lu, "a_3^{-1}"] \arrow[lld, "a_1"] & {(v_4,-1)} \arrow[l, "b^{-1}"]                      &                                   &                                     \\
{(v_2,-1)} \arrow[uuuuu, "a_2^{-1}"] &                                  &                                                      &                                                     &                                   & {(v_6,1)} \arrow[llu, "c_3"]       
\end{tikzcd}
    \caption{Covering quiver for $\Gamma''$}
    \label{fig:covering quiver}
\end{figure}
\end{exmp}

Using the covering quiver, we define a relation $\prec$ on $\cq{0}$ as follows. 

Given $(v_1,i_1),(v_2,i_2)\in\cq{0}$, say $(v_1,i_1)\prec(v_2,i_2)$ if there is $\alpha\in Q_1$ such that $(v_2,i_2)\xrightarrow{\alpha}(v_1,i_1)$ or $(v_1,i_1)\xrightarrow{\alpha^{-1}}(v_2,i_2)$.

We also denote the transitive closure of $\prec$ by the same symbol. The following proposition shows that $\prec$ is a strict partial order on $\cq{0}$ for the string algebras that we are concerned with in this paper.

\begin{prop}\label{prop: poset from cov quiv}
Let $\Lambda$ be a string algebra whose underlying quiver is acyclic. Then the relation $\prec$ on $\cq{0}$ defined above is anti-symmetric.
\end{prop}
\begin{proof}
If not, then there exist $(v,i),(v',i')\in\cq{0}$ such that $(v,i)\prec(v',i')$ and  $(v',i')\prec(v,i)$. Figure \ref{fig:cov quiver antisymm} below demonstrates the situation. Corresponding to $(v',i')\prec(v,i)$, by transitivity of $\prec$, there is some $k\geq1$ such that for each $m\in\{1,\cdots,k\}$, there exists exactly one of $\alpha_m\in Q_1$ or $\beta_m^{-1}\in Q_1^{-1}$ witnessing $(v_{m-1},i_{m-1})\succ(v_m,i_m)$, where $(v_0,i_0):=(v,i)$ and $(v_k,i_k):=(v',i')$. Similarly, corresponding to $(v,i)\prec(v',i')$, there is some $l\geq1$ such that for each $j\in\{1,\cdots,l\}$, there exists exactly one of $\alpha'_j\in Q_1$ or ${\beta'_j}^{-1}\in Q_1^{-1}$ witnessing $(v'_{j-1},i_{j-1})\prec(v'_j,i_j)$, where $(v'_0,j_0):=(v,i)$ and $(v'_l,i_l):=(v',i')$.

\begin{figure}[h]
\centering

\begin{minipage}{0.5\textwidth}
\centering
\[\begin{tikzcd}[column sep=small]
                                                                                           & {(v_1,i_1)} \arrow[r, "\alpha_2", shift left=2] \arrow[ld, "\beta_1^{-1}", shift left=2]           & {(v_2,i_2)} \arrow[r, no head, dashed] \arrow[l, "\beta_2^{-1}", shift left=2]  & {(v_{k-1},i_{k-1})} \arrow[rd, "\alpha_k", shift left=2]            &                                                                                            \\
{(v,i)} \arrow[ru, "\alpha_1", shift left=2] \arrow[rd, "{\beta'}_1^{-1}"', shift right=2] &                                                                                                    &                                                                                 &                                                                     & {(v',i')} \arrow[lu, "\beta_k^{-1}", shift left=2] \arrow[ld, "\alpha'_l"', shift right=2] \\
                                                                                           & {(v'_1,i'_1)} \arrow[r, "{\beta'}_2^{-1}"', shift right=2] \arrow[lu, "\alpha'_1"', shift right=2] & {(v'_2,i'_2)} \arrow[r, no head, dashed] \arrow[l, "\alpha'_2"', shift right=2] & {(v'_{l-1},i'_{l-1})} \arrow[ru, "{\beta'}_l^{-1}"', shift right=2] &                                                                                           
\end{tikzcd}\]
\caption{}
\label{fig:cov quiver antisymm}
\end{minipage}%
\begin{minipage}{0.5\textwidth}
\centering
\[\begin{tikzcd}[column sep=small]
                         & v_1 \arrow[r, "\gamma_2"]    & v_2 \arrow[r, no head, dashed]                         & v_{k-1} \arrow[rd, "\gamma_k"] &                            \\
v \arrow[ru, "\gamma_1"] &                              &                                                        &                                & v' \arrow[ld, "\gamma'_l"] \\
                         & v'_1 \arrow[lu, "\gamma'_1"] & v'_2 \arrow[l, "\gamma'_2"] \arrow[r, no head, dashed] & v'_{l-1}                       &                           
\end{tikzcd}\]
\caption{}
\label{fig:underlying quiver of lambda antisymm}
\end{minipage}

\end{figure}
Therefore the underlying quiver of $\Lambda$ contains the quiver demonstrated in Figure \ref{fig:underlying quiver of lambda antisymm} as a subquiver, where $\gamma_i\in\{\alpha_i,\beta_i\}$ for each $i\in\{1,\cdots,k\}$ and  $\gamma'_j\in\{\alpha'_j,\beta'_j\}$ for each $j\in\{1,\cdots,l\}$, which is a contradiction to the hypothesis that the underlying quiver of $\Lambda$ is acyclic.
\end{proof}

\begin{exmp}\label{exmp: strict poset}
Continuing from Example \ref{exmp: covering quiver}, Figure \ref{fig:hasse diagram} shows the Hasse diagram for the strict partial order $(\overline{\C Q}_0(\Gamma''),\prec)$.
\end{exmp}

\begin{figure}[h]
    \centering
    \begin{tikzcd}[row sep=small,column sep=small]
{(v_3,1)}                                         &                                                      &                                                      & {(v_3,-1)}                                         \\
{(v_1,1)} \arrow[u, no head]                      &                                                      &                                                      & {(v_1,-1)} \arrow[u, no head]                      \\
                                                  & {(v_2,1)} \arrow[lu, no head] \arrow[rruu, no head]  & {(v_2,-1)} \arrow[lluu, no head] \arrow[ru, no head] &                                                    \\
                                                  & {(v_6,-1)} \arrow[ld, no head] \arrow[rrdd, no head] & {(v_6,1)} \arrow[lldd, no head] \arrow[rd, no head]  &                                                    \\
{(v_4,1)} \arrow[d, no head] \arrow[uuu, no head] &                                                      &                                                      & {(v_4,-1)} \arrow[d, no head] \arrow[uuu, no head] \\
{(v_5,1)}                                         &                                                      &                                                      & {(v_5,-1)}                                        
\end{tikzcd}
    \caption{Hasse Diagram of $(\overline{\C Q}_0(\Gamma''),\prec)$}
    \label{fig:hasse diagram}
\end{figure}

A similar argument as in the proof of Proposition \ref{prop: poset from cov quiv} gives the following result.

\begin{cor}\label{cor: empty trace always exists}
Let $\Lambda$ be a string algebra whose underlying quiver is acyclic. Let $(\cq0,\prec)$ be the strict partial order as obtained above. Then for every $(v,i)\in\cq0$, we have that $(v,i)$ and $(v,-i)$ are incomparable.
\end{cor}

Now that we have obtained a partially ordered set from the covering quiver, we have to trace appropriate paths in the partially ordered set. In view of Proposition \ref{prop:string at most one sequence}, given $(v,i),(v',i')\in\cq{0}$, there exists at most one string $\f z=\alpha_k\cdots\alpha_1$ with $\delta(\f z)\neq0$ such that $1_{(v',i')}\f z1_{(v,i)}$ is a string. Define 
\begin{equation*}
T_{((v,i),(v',i'))}:=
\begin{cases}
\langle(v,i),(s(\alpha_2),-\varsigma(\alpha_2)),\cdots,(s(\alpha_k),-\varsigma(\alpha_k)),(v',i')\rangle&\text{if such a string $\f z$ exists;}\\
\langle\rangle &\text{otherwise.}
\end{cases}.
\end{equation*}

Let $\C T$ be the collection of all such sequences $T_{((v,i),(v',i'))}$. The set $\C T$ contains the empty sequence, thanks to Corollary \ref{cor: empty trace always exists}, and Proposition \ref{prop:string at most one sequence} ensures that it is closed under subsequences. Also define $\upmu:A\to A$ by $\upmu(v,i):=(v,-i)$, where $(v,i)\in A$.

\begin{exmp}\label{exmp: tracings}
Continuing from Example \ref{exmp: strict poset}, the following are the maximal sequences in $\C T$:
$\\\{\langle(v_1,-1),(v_2,-1)\rangle,\langle(v_2,-1),(v_3,1)\rangle,\langle(v_3,1),(v_1,1),(v_4,1)\rangle,\langle(v_1,1),(v_4,1),(v_5,1)\rangle,\\\langle(v_5,1),(v_6,1)\rangle,\langle(v_6,1),(v_4,-1)\rangle,\langle(v_4,-1),(v_1,-1),(v_3,-1)\rangle,\langle(v_3,-1),(v_2,1)\rangle,(v_2,1),(v_1,1)\rangle\}.$
\end{exmp}

The next result which summarizes the discussion in this section provides abundant examples of finite traced posets.

\begin{prop}\label{prop: traced poset obtained from covering quiver}
The quadruple $(\cq{0},\prec,\C T,\upmu)$ is a finite traced partial ordered set.
\end{prop}
\begin{proof}
We already have that $(\cq{0},\prec)$ is a finite strict partial ordered set, thanks to Proposition \ref{prop: poset from cov quiv}, and that $\C T$ is a collection of strictly monotone sequences closed under (possibly empty) subsequences. Now we verify the remaining conditions of Definition \ref{defn: traced poset}.

Properties \ref{at most two elements cover} and \ref{at most two elements cover}$'$ follow from Condition \ref{indegoutdeg} of Definition \ref{defn:string alg}. Property \ref{at most one sequence} follows from Proposition \ref{prop:string at most one sequence}. The definition of $T_{((v,i),(v',i'))}$ forces Properties \ref{successor is a cover} and \ref{successor is a cover}$'$. Properties \ref{at most one essential neighbour} and \ref{at most one essential neighbour}$'$ follow from Condition \ref{at most one arrow} of Definition \ref{defn:string alg}. It is clear from the definition of $\upmu$ that $\upmu$ is an involution, isomorphism, and has no fixed points. For Property \ref{T and mu}, note that if $1_{(v',i')}\f u1_{(v,i)}$ is a string, where $(v,i),(v',i')\in \cq{0}$, then $1_{(v,-i)}\f u^{-1}1_{(v',-i')}$ is a string.
\end{proof}

\section{Recovering a string algebra from its finite traced poset}\label{sec:construct string alg}
After obtaining a finite traced poset from a string algebra whose underlying quiver is acyclic, a natural question is how much information we lose in the process. In other words, given a finite traced poset, how much of the string algebra can be recovered? It turns out that we do not lose any information in this process. Given a finite traced poset $\s A$, we can construct a string algebra $\Lambda$ such that the traced poset constructed from $\Lambda$ is isomorphic to $\s A$. In this section, we describe the process of constructing a string algebra from a finite traced poset. We omit the routine verification that the traced poset constructed from $\Lambda$ is isomorphic to $\s A$.

Let $\s A:=(A,<,T,\mu)$ be a finite traced partial ordered set. Construct a quiver $Q=(Q_0,Q_1,s,t)$, where $$Q_0:=\{\{n,\mu(n)\}\mid n\in\s A\}.$$ Consider $\{n,\mu(n)\},\{m,\mu(m)\}\in Q_0$ and fix $n'\in\{n,\mu(n)\}$. For each instance of the following conditions, there exists an arrow $\alpha\in Q_1$ with $s(\alpha)=\{n,\mu(n)\}$ and $t(\alpha)=\{m,\mu(m)\}$:
\begin{enumerate}[label=$(\roman*)$]
    \item $n'$ covers $m$ and $T_{(n',m)}\neq\langle\rangle$,
    \item $n'$ covers $\mu(m)$ and $T_{(n',\mu(m))}\neq\langle\rangle$,
    \item $n'$ covers $m$ and $T_{(m,n')}\neq\langle\rangle$,
    \item $n'$ covers $\mu(m)$ and $T_{(\mu(m),n')}\neq\langle\rangle$.
\end{enumerate}

\begin{rmk}\label{rmk: atmost two out of four (construction of str alg from poset)}
Axioms \ref{at most one essential neighbour} and \ref{at most one essential neighbour}$'$ of Definition \ref{defn: traced poset} imply that at most one condition from $(i)$ and $(ii)$ and at most one condition from $(iii)$ and $(iv)$ can be satisfied.
\end{rmk}

By our construction of the quiver $Q$, if there is a path $$v_1\to v_2\to\cdots\to v_k$$ then there are $n_i\in v_i$ in $Q$ such that $$n_1> n_2>\cdots>n_k.$$ Let $\rho$ contain such a path if one of the following holds:
\begin{itemize}
    \item $T_{(n_1,n_k)}=\langle\rangle$ and $T_{(n_i,n_j)}\neq\langle\rangle$ for $1\leq i<j\leq k$ and $(i,j)\neq(1,k)$,

    \item $T_{(n_k,n_1)}=\langle\rangle$ and $T_{(n_i,n_j)}\neq\langle\rangle$ for $1\leq j<i\leq k$ and $(i,j)\neq(k,1)$.
\end{itemize}

\begin{prop}
The pair $(Q,\rho)$ constructed above is a presentation of a string algebra with acyclic underlying quiver.
\end{prop}
\begin{proof}
Remark \ref{rmk: atmost two out of four (construction of str alg from poset)} ensures Condition \ref{indegoutdeg} of Definition \ref{defn:string alg}. Moreover, Conditions \ref{successor is a cover}, \ref{successor is a cover}$'$, \ref{at most one essential neighbour} and \ref{at most one essential neighbour}$'$ of Definition \ref{defn: traced poset} together ensure Condition \ref{at most one arrow}. Finally, since $(A,<)$ is a finite partial order, there is no directed cycle in $Q$, thereby proving Condition \ref{admissible} as well.
\end{proof}

\section{Band bricks over string algebras whose underlying quiver is acyclic} \label{sec: main}
Now we are in a position to characterise band bricks over string algebras with acyclic underlying quiver in terms of weakly perfectly clustering pairs of crowns over a finite traced poset. The initial part of this section establishes one-to-one correspondences between certain sets of words and strings, which provide us with a way of viewing bands as certain words over a finite traced poset and vice versa. The latter part of this section is devoted to the proof of Theorem \ref{thm: main}.

Given a string algebra $\Lambda$ with acyclic underlying quiver, let $\s A:=(A,<,T,\mu):=(\cq{0},\prec,\C T,\upmu)$ as obtained in \S~\ref{sec:obtain traced poset}. Recall the definitions of a valid zigzag (Definition \ref{defn: valid zigzag}) and a valid crown (Definition \ref{defn: valid crown2}).

Define a map $\C{W_{\s{St}}}:\s{St}(\Lambda)\xrightarrow{}\ZZval$ as follows. If $\f x$ is a string with $|\f x|=0$ then set $\wst{\f x}:=p\in A$ such that $1_p=\f x$. On the other hand, given a string $\f x$ with $|\f x|>0$, let its standard partition be $\f x=\f x_k\cdots\f x_1$. By Remark \ref{rem:alt string repn}, there exist $n_1,\cdots,n_{k+1}\in \s A$ such that $\f x=1_{n_{k+1}}\f x_k1_{n_k}\cdots1_{n_2}\f x_11_{n_1}$. Define $\wst{\f x}:=n_1n_2\cdots n_{k+1}$. Since $\delta(\f x_i)=-\delta(\f x_{i+1})$ for $i\in\{1,\cdots,k-1\}$, it follows that $\wst{\f x}$ is indeed a zigzag. Moreover, for each $i\in\{1,\cdots,k\}$, since $1_{n_{i+1}}\f x_{i+1}1_{n_i}\in \s{St}(\Lambda)$ with $\delta(\f x_{i+1})\neq0$, we conclude that $T_{(n_i,n_{i+1})}\neq\emptyset$ and hence $\wst{\f x}\in\ZZval$.

Note that if $\f b$ is a special cyclic permutation of a band then the first and last letter of $\C{W_{\s{St}}}(\f b)$ are identical. This induces a map $\C W_{\s{Ba}}:\cycsp\to\pcrownsval$, where $\C{W_{\s{Ba}}}(\f b)$ is defined as the word obtained from $\C{W_{\s{St}}}(\f b)$ by omitting the last letter; the primitivity of $\f b$ forces the primitivity of $\wba{\f b}$.

Now we define maps in the other direction. Given a valid zigzag $\s w=n_1\cdots n_k$, we know from the construction of $T=\C T$ that there exists a unique string $\f x_i$ with $\delta(\f x_i)\neq 0$ such that $1_{n_{i+1}}\f x_i1_{n_i}$ is a string for each $i\in\{1,\cdots,k\}$. Define $\overline{\C W}_\s{St}:\ZZval\to\s{St}(\Lambda)$ by $\overline{\C W}(\s w):=\f x_{k-1}\cdots\f x_1$. Similarly, if $\s w=n_1\cdots n_k$ is a primitive crown then $\s w':=n_1\cdots n_kn_1$ is a valid zigzag. Define $\overline{\C W}_{\s{Ba}}:\pcrownsval\to\s{Cyc}^{\mathrm{Sp}}(\Lambda)$ by $\overline{\C W}_{\s{Ba}}(\s w):=\overline{\C W}_{\s{St}}(n_1\cdots n_kn_1)$.

The next result summarizes all the one-to-one correspondences.

\begin{prop}\label{prop:bijections}
Let $\Lambda$ be a string algebra whose underlying quiver is acyclic. Then we have the following one-to-one correspondences.

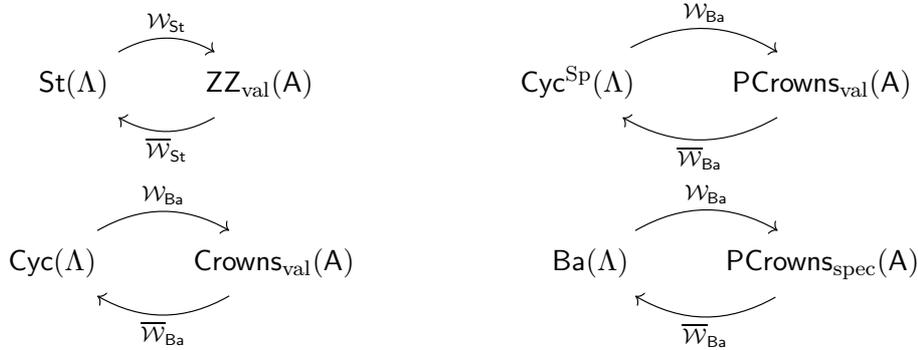
\hspace{2cm}
\begin{tikzcd}
\s{St}(\Lambda) \arrow[r, "\C W_\s{St}", bend left, shift left] & \ZZval \arrow[l, "\overline{\C W}_\s{St}", bend left, shift left]
\end{tikzcd}
\hfill
\begin{tikzcd}
\s{Cyc}^{\mathrm{Sp}}(\Lambda) \arrow[r, "\C W_\s{Ba}", bend left, shift left] & \pcrownsval \arrow[l, "\overline{\C W}_\s{Ba}", bend left, shift left]
\end{tikzcd}
\hspace{2cm}

\hspace{1.6cm}
\begin{tikzcd}
\cyc\arrow[r, "\C W_\s{Ba}", bend left, shift left] & \crownsval \arrow[l, "\overline{\C W}_\s{Ba}", bend left, shift left]
\end{tikzcd}
\hfill
\begin{tikzcd}
\s{Ba}(\Lambda) \arrow[r, "\C W_\s{Ba}", bend left, shift left] & \pcrownsspec \arrow[l, "\overline{\C W}_\s{Ba}", bend left, shift left]
\end{tikzcd}
\hspace{1.9cm}

\end{prop}


Now, given bands $\f b_1$ and $\f b_2$, we give a criterion for the existence of a non-trivial morphism from $B(\f b_1,1,\lambda_1)$ to $B(\f b_2,1,\lambda_2)$, where $\lambda_1,\lambda_2\in\C K^\times$.

\begin{lem}\label{lem: nontrivial morphism between bands criterion in terms of WPCC}
Given bands $\f b_1$ and $\f b_2$, $\lambda_1,\lambda_2\in\C K^\times$, there does not exist a non-trivial morphism from $B(\f b_1,1,\lambda_1)$ to $B(\f b_2,1,\lambda_2)$ if and only if the pairs $(\wba{\f b_1},\wba{\f b_2})$ and $(\wba{\f b_1},\wba{\f b_2^{-1}})$ are weakly perfectly clustering.
\end{lem}
\begin{proof}
$(\Longleftarrow)$ Suppose there exists a non-trivial morphism $B(\f b_1,1,\lambda_1)\to B(\f b_2,1,\lambda_2)$. By Corollary \ref{cor:existence of morphisms between band modules}, without loss, there exists a finite string $\f u$ that occurs as a factor substring of $\INF{\f b_1}$ and an image substring of $\INF{\f b_2}$.

We claim that we can choose $\f u$ such that either $|\f u|\leq|\f b_1|$ or $|\f u|\leq|\f b_2|$. If possible, suppose $|\f u|>|\f b_1|$ and $|\f u|>|\f b_2|$. Then there exist cyclic permutations $\f b_1'$ and $\f b_2'$ of $\f b_1$ and $\f b_2$ respectively such that $\f b_1'\sqsubset_l\f u$ and $\f b_2'\sqsubset_l\f u$. Hence either $\f b'_1=\f b'_2$ or $|\f b'_1|\neq|\f b'_2|$. The former case contradicts that $\f u$ is both a factor and image substring of $\INF{\f b}$. In the latter case, note that both $|\f b_1|$ and $|\f b_2|$ are periods (Definition \ref{defn: period of a string}) of $\f u$, which implies, by Proposition \ref{prop: period of a string}, that $\mathrm{gcd}(|\f b_1|,|\f b_2|)$ is a period of $\f u$. Without loss, suppose $|\f b_1|>|\f b_2|$. Since $\f b'_1\sqsubset_l\f u$ and $\mathrm{gcd}(|\f b_1|,|\f b_2|)$ divides $|\f b'_1|=|\f b_1|$, we have that $\mathrm{gcd}(|\f b_1|,|\f b_2|)(<|\f b_1|)$ is a period of $\f b'_1$, which contradicts that $\f b'_1$ is a primitive cyclic string.

Therefore, without loss, let $|\f u|<|\f b_1|$. Since $\f u$ is a factor substring of $\INF{\f b_1}$, let $\f x_1$ and $\f y_1$ be maximal direct strings such that $\f y_1\f u\f x_1^{-1}\sqsubset_l\INF{\f b_1}$. It is not difficult to note that $|\f y_1\f u\f x_1^{-1}|\leq|\f b_1^2|$, which gives $\f y_1\f u\f x_1^{-1}\sqsubset\f b_1^3$. Therefore, we have $\f y_1\f u\f x_1^{-1}\sqsubset\f b_1^{2\cdot\wba{\f b_2}}$ since $\wba{\f b_2}\geq 2$, which also gives $|\wst{\f y_1\f u\f x_1^{-1}}|<2|\wba{\f b_1}||\wba{\f b_2}|$.

Since $\f u$ is an image substring of $\INF{\f b_2}$, let $\f x_2$ and $\f y_2$ be maximal direct strings such that $\f y_2^{-1}\f u\f x_2\sqsubset\INF{\f b_2}$. Note that $|\wst{\f y_2^{-1}\f u\f x_2}|=|\wst{\f y_1\f u\f x_1^{-1}}|<2|\wba{\f b_1}||\wba{\f b_2}|$. Therefore, we have $\f y_2^{-1}\f u\f x_2\sqsubset\f b_2^{2\cdot\wba{\f b_1}}$.

Now there are two cases.

\textbf{Case 1.} $|\f u|=0$.

Let $p:=(v,i)\in\s A$ be such that $\f u=1_{(v,i)}$. Then $\wst{\f y_1\f u\f x_1^{-1}}=n'pm'$ and $\wst{\f y_2^{-1}\f u\f x_2}=n''pm''$, where $n',m'<p$ and $n'',m''>p$. Therefore we have that there is a cyclic permutation $n'pm'\s z_1$ of $\wba{\f b_1}^{2\cdot\wba{\f b_2}}$ and a cyclic permutation  $n''pm''\s z_2$ of $\wba{\f b_2}^{2\cdot\wba{\f b_1}}$ such that $n'<n''$ and $m'<m''$. By Condition \ref{axiom1 of WPCC2} of Definition \ref{defn: weak perf clus2}, we have that the pair $(\wba{\f b_1},\wba{\f b_2})$ is not weakly perfectly clustering.\\

\textbf{Case 2.} $|\f u|>0$

By Remark \ref{rem:alt string repn}, we have $\f y_1\f u\f x_1^{-1}=\f y_11_m\f u1_n\f x_1^{-1}$ and $\f y_2^{-1}\f u\f x_2=\f y_2^{-1}1_m\f u1_n\f x_2$ for appropriate $m,n\in \s A$. Now we have three subcases.\\

\textbf{Subcase 1.} $\delta{(\f u)}=0$.

There will be four cases based on the signs of the first and last syllables of $\f u$. We will deal with only one of the cases; the rest can be handled similarly. Assume that the first syllable of $\f u$ is inverse and the last syllable is direct. We have $\wst{\f y_11_m\f u1_n\f x_1^{-1}}=n'\s zm'$ and $\wst{1_m\f u1_n}=n\s zm$ for some $\s z\in\s A^+$, where $n'<n$ and $m'<m$. Therefore we obtain a cyclic permutation $n'\s zm'\s z'$ of $\wba{\f b_1}^{2\cdot\wba{\f b_2}}$ and a cyclic permutation  $n\s zm\s z''$ of $\wba{\f b_2}^{2\cdot\wba{\f b_1}}$ such that $n'<n$ and $m'<m$. By Condition \ref{axiom1 of WPCC2} of Definition \ref{defn: weak perf clus2}, we have that $(\wba{\f b_1},\wba{\f b_2})$ is not weakly perfectly clustering.\\

\textbf{Subcase 2.} $\delta{(\f u)}=1$.

In this case, we have $n<m$. So $\wst{1_m\f u1_n\f x_1^{-1}}=n'm$ and $\wst{\f y_2^{-1}1_m\f u1_n}=nm''$, where $n'<n$ and $m<m''$. Therefore we obtain a cyclic permutation $n'm\s z'$ of $\wba{\f b_1}^{2\cdot\wba{\f b_2}}$ and a cyclic permutation  $nm''\s z''$ of $\wba{\f b_2}^{2\cdot\wba{\f b_1}}$ such that $n'<n$ and $m<m''$. Note that $n\in T_{(n',m)}\cap T_{(n,m'')}$. Thus by Condition \ref{axiom2 of WPCC2} of Definition \ref{defn: weak perf clus2}, we have that $(\wba{\f b_1},\wba{\f b_2})$ is not weakly perfectly clustering.\\

\textbf{Subcase 3.} $\delta{(\f u)}=-1$.

In this case, we have $n>m$. So $\wst{\f y_11_m\f u1_n}=nm'$ and $\wst{1_m\f u1_n\f x_2}=n''m$, where $n<n''$ and $m'<m$. Therefore we obtain a cyclic permutation $nm'\s z'$ of $\wba{\f b_1}^{2\cdot\wba{\f b_2}}$ and a cyclic permutation  $n''m\s z''$ of $\wba{\f b_2}^{2\cdot\wba{\f b_1}}$ such that $n<n''$ and $m''<m$. Note that $m\in T_{(n',m)}\cap T_{(n,m'')}$. Thus again by Condition \ref{axiom2 of WPCC2} of Definition \ref{defn: weak perf clus2}, we have that $(\wba{\f b_1},\wba{\f b_2})$ is not weakly perfectly clustering.\\

$(\Longrightarrow)$ Without loss, suppose the pair $(\s w_1:=\wba{\f b_1},\s w_2:=\wba{\f b_2})$ is not weakly perfectly clustering. There are two cases as per Definition \ref{defn: weak perf clus2}.

\textbf{Case 1.} There exist cyclic permutations $\s w'$ of $\s w_1^{2\cdot|\s w_2|}$ and $\s w''$ of $\s w_2^{2\cdot|\s w_1|}$ such that $\s w'=n'\s zm'\s z'$ and $\s w''=n''\s zm''\s z''$, where $n'<n'',\ m'<m''$ and $|\s z|\geq 1$.

Let $\f x_1:=\winvst{\s w'}$ and $\f x_2:=\winvst{\s w''}$. It is clear that $\f x_1\sqsubset\INF{\f b_1}$ and $\f x_2\sqsubset\INF{\f b_2}$.\\

\textbf{Subcase 1.} $|\s z|>1$.

Let $\s z:=p_1\overline{\s z}p_2$, where $|\overline{\s z}|\geq 0$. Since $\s w'$ and $\s w''$ are zigzags, we have $n'<p_1$ if and only if $n''<p_1$. Similarly, $m'<p_2$ if and only if $m''<p_2$. Figure \ref{fig:case1 Subcase1} shows the different possibilities that can occur.

\begin{figure}[h]
    \hspace{2mm}
    \begin{tikzpicture}[x=0.4pt,y=0.4pt,yscale=-1,xscale=1]
 
\draw   (20,70) -- (500,70) -- (500,270) -- (20,270) -- cycle ;
\draw    (47.71,190) -- (86.24,100) ;
\draw    (86.24,100) -- (105.5,140) ;
\draw    (163.3,150) -- (192.2,100) ;
\draw    (230.73,210) -- (192.2,100) ;
\draw [line width=1.5]  [dash pattern={on 5.63pt off 4.5pt}]  (105.5,120) -- (163.3,120) ;
\draw    (310,170) -- (335.83,100) ;
\draw    (335.83,100) -- (355.1,140) ;
\draw    (412.9,150) -- (441.8,100) ;
\draw    (470,180) -- (441.8,100) ;
\draw [line width=1.5]  [dash pattern={on 5.63pt off 4.5pt}]  (355.1,120) -- (412.9,120) ;

\draw (160,280.4) node [anchor=north west][inner sep=0.75pt]    {$n'< p_{1} ,\ m'< p_{2}$};
\draw (40.67,192.4) node [anchor=north west][inner sep=0.75pt]    {$n'$};
\draw (229.4,212.4) node [anchor=north west][inner sep=0.75pt]    {$m'$};
\draw (78.39,142.4) node [anchor=north west][inner sep=0.75pt]    {$\underbrace{\ \ \ \ \ \ \ \ \ \ \ \ }$};
\draw (130.33,166.4) node [anchor=north west][inner sep=0.75pt]    {$\mathsf{z}$};
\draw (301,172.4) node [anchor=north west][inner sep=0.75pt]    {$n''$};
\draw (455,182.4) node [anchor=north west][inner sep=0.75pt]    {$m''$};
\draw (327.98,142.4) node [anchor=north west][inner sep=0.75pt]    {$\underbrace{\ \ \ \ \ \ \ \ \ \ \ \ }$};
\draw (379.93,166.4) node [anchor=north west][inner sep=0.75pt]    {$\mathsf{z}$};
\end{tikzpicture}
    \hfill
    \begin{tikzpicture}[x=0.4pt,y=0.4pt,yscale=-1,xscale=1]

\draw   (20,70) -- (500,70) -- (500,270) -- (20,270) -- cycle ;
\draw    (42.71,227) -- (81.24,137) ;
\draw    (81.24,137) -- (100.5,177) ;
\draw    (190,190) -- (175.2,137) ;
\draw [line width=1.5]  [dash pattern={on 5.63pt off 4.5pt}]  (99.5,159) -- (170,160) ;
\draw    (310,190) -- (330.83,137) ;
\draw    (330.83,137) -- (350.1,177) ;
\draw    (190,190) -- (220,130) ;
\draw    (438,190) -- (418.2,141) ;
\draw    (438,190) -- (470,100) ;
\draw [line width=1.5]  [dash pattern={on 5.63pt off 4.5pt}]  (350,160) -- (420,160) ;

\draw (160,280.4) node [anchor=north west][inner sep=0.75pt]    {$n'< p_{1} ,\ m' >p_{2}$};
\draw (35.67,229.4) node [anchor=north west][inner sep=0.75pt]    {$n'$};
\draw (218,98.4) node [anchor=north west][inner sep=0.75pt]    {$m'$};
\draw (82.39,195.4) node [anchor=north west][inner sep=0.75pt]    {$\underbrace{\ \ \ \ \ \ \ \ \ \ \ }$};
\draw (131.33,220.4) node [anchor=north west][inner sep=0.75pt]    {$\mathsf{z}$};
\draw (291,192.4) node [anchor=north west][inner sep=0.75pt]    {$n''$};
\draw (450,72.4) node [anchor=north west][inner sep=0.75pt]    {$m''$};
\draw (330.98,195.4) node [anchor=north west][inner sep=0.75pt]    {$\underbrace{\ \ \ \ \ \ \ \ \ \ \ }$};
\draw (379.93,220.4) node [anchor=north west][inner sep=0.75pt]    {$\mathsf{z}$};
\end{tikzpicture}
    \\
    \vspace{4mm}
    \hspace{2mm}
    \begin{tikzpicture}[x=0.4pt,y=0.4pt,yscale=-1,xscale=1]

\draw   (20,70) -- (500,70) -- (500,270) -- (20,270) -- cycle ;
\draw    (80,180) -- (96.57,130) ;
\draw    (163.63,180) -- (192.53,130) ;
\draw    (231.06,240) -- (192.53,130) ;
\draw [line width=1.5]  [dash pattern={on 5.63pt off 4.5pt}]  (105.83,150) -- (163.63,150) ;
\draw    (80,180) -- (60,120) ;
\draw    (319.27,180) -- (335.83,130) ;
\draw    (402.9,180) -- (431.8,130) ;
\draw    (460,210) -- (431.8,130) ;
\draw [line width=1.5]  [dash pattern={on 5.63pt off 4.5pt}]  (345.1,150) -- (402.9,150) ;
\draw    (319.27,180) -- (290,90) ;

\draw (160,280.4) node [anchor=north west][inner sep=0.75pt]    {$n' >p_{1} ,\ m'< p_{2}$};
\draw (41,90.4) node [anchor=north west][inner sep=0.75pt]    {$n'$};
\draw (229.73,242.4) node [anchor=north west][inner sep=0.75pt]    {$m'$};
\draw (78.72,190.4) node [anchor=north west][inner sep=0.75pt]    {$\underbrace{\ \ \ \ \ \ \ \ \ \ \ \ }$};
\draw (130.66,215.4) node [anchor=north west][inner sep=0.75pt]    {$\mathsf{z}$};
\draw (293,72.4) node [anchor=north west][inner sep=0.75pt]    {$n''$};
\draw (451,222.4) node [anchor=north west][inner sep=0.75pt]    {$m''$};
\draw (317.98,190.4) node [anchor=north west][inner sep=0.75pt]    {$\underbrace{\ \ \ \ \ \ \ \ \ \ \ \ }$};
\draw (369.93,215.4) node [anchor=north west][inner sep=0.75pt]    {$\mathsf{z}$};
\end{tikzpicture}
    \hfill
    \begin{tikzpicture}[x=0.4pt,y=0.4pt,yscale=-1,xscale=1]

\draw   (20,70) -- (500,70) -- (500,270) -- (20,270) -- cycle ;
\draw    (190,140) -- (167.51,212.3) ;
\draw    (167.51,212.3) -- (148.44,172.2) ;
\draw    (90.69,161.91) -- (61.54,211.76) ;
\draw    (44,156) -- (61.54,211.76) ;
\draw [line width=1.5]  [dash pattern={on 5.63pt off 4.5pt}]  (148.34,192.2) -- (90.54,191.91) ;
\draw    (460,110) -- (425.51,212.3) ;
\draw    (425.51,212.3) -- (406.44,172.2) ;
\draw    (348.69,161.91) -- (319.54,211.76) ;
\draw    (290,110) -- (319.54,211.76) ;
\draw [line width=1.5]  [dash pattern={on 5.63pt off 4.5pt}]  (406.34,192.2) -- (348.54,191.91) ;

\draw (160,280.4) node [anchor=north west][inner sep=0.75pt]    {$n' >p_{1} ,\ m' >p_{2}$};
\draw (60.59,213.68) node [anchor=north west][inner sep=0.75pt]  [rotate=-359.77]  {$\underbrace{\ \ \ \ \ \ \ \ \ \ \ }$};
\draw (120.91,257.59) node [anchor=north west][inner sep=0.75pt]  [rotate=-180.29]  {$\mathsf{z}$};
\draw (33,120.4) node [anchor=north west][inner sep=0.75pt]    {$n'$};
\draw (175,110.4) node [anchor=north west][inner sep=0.75pt]    {$m'$};
\draw (318.59,213.68) node [anchor=north west][inner sep=0.75pt]  [rotate=-359.77]  {$\underbrace{\ \ \ \ \ \ \ \ \ \ \ }$};
\draw (378.91,257.59) node [anchor=north west][inner sep=0.75pt]  [rotate=-180.29]  {$\mathsf{z}$};
\draw (283,75.4) node [anchor=north west][inner sep=0.75pt]    {$n''$};
\draw (440,75.4) node [anchor=north west][inner sep=0.75pt]    {$m''$};
\end{tikzpicture}
    \caption{Different possibilities for Case 1.1}
    \label{fig:case1 Subcase1}
\end{figure}
Let $\overline{n}:=
\begin{cases}
    n'' &\text{ if }n'<p_1;\\
    n' &\text{ if }n'>p_1,
\end{cases}
\text{ and\hspace{10pt}}
\overline{m}:=
\begin{cases}
    m'' &\text{ if }m'<p_2;\\
    m' &\text{ if }m'>p_2.\\
\end{cases}
$

Note that $\overline{n}\s z\overline{m}$ is a valid zigzag. We claim that $\f u:=\winvst{\overline{n}\s z\overline{m}}$ is a string which appears as a factor substring of $\INF{\f b_1}$ and an image substring of $\INF{\f b_2}$.

Since $\f u$ is a substring of $\f x_1$ and $\f x_2$, consider the first occurrence (from the right, as per the convention for reading strings) of $\f u$ in $\winvba{\s w'}$, say $\alpha_{k'+l}\cdots\alpha_{k'+1}$, and in $\winvba{\s w''}$, say $\alpha_{k''+l}\cdots\alpha_{k''+1}$. Consider the marked strings $\dot{\f y}_1:=(\winvba{\s w'},k',k'+l)$ and $\dot{\f y}_2:=(\winvba{\s w''},k'',k''+l)$ (Definition \ref{defn: marked string}). These marked strings can be extended to marked infinite strings $\INF{\dot{\f y_1}}$ and $\INF{\dot{\f y_2}}$ in an obvious way. It is straightforward to note that $\vartheta_l(\INF{\dot{\f y_1}})=-\vartheta_l(\INF{\dot{\f y_2}})=\vartheta_r(\INF{\dot{\f y_1}})=-\vartheta_r(\INF{\dot{\f y_2}})=-1$, i.e. $\f u$ is a factor substring of $\INF{\winvba{\s w'}}=\INF{\f b_1}$ and an image substring of $\INF{\winvba{\s w''}}=\INF{\f b_2}$. Thus by Corollary \ref{cor:existence of morphisms between band modules}, there exists a non-trivial morphism $B(\f b_1,1,\lambda_1)\to B(b_2,1,\lambda_2)$.
\\

\textbf{Subcase 2.} $|\s z|=1.$

Let $\s z:=p$. So we have $n'<p$ if and only if $m'<p$, and $n''<p$ if and only if $m''<p$. Figure \ref{fig:case1 subcase2} shows three different possibilities in this case. The first two of them can be dealt with like Case 1.1 above.

\begin{figure}[h]
    \hspace{2mm}
    \begin{tikzpicture}[x=0.4pt,y=0.4pt,yscale=-1,xscale=1]

\draw   (20,70) -- (500,70) -- (500,270) -- (20,270) -- cycle ;
\draw    (110,200) -- (150.41,130) ;
\draw    (220,220.32) -- (150.41,130) ;
\draw    (370,169.26) -- (400.97,130) ;
\draw    (439.73,180.96) -- (400.97,130) ;

\draw (170,280.4) node [anchor=north west][inner sep=0.75pt]    {$n'< p ,\ n''< p$};
\draw (91,212.4) node [anchor=north west][inner sep=0.75pt]    {$n'$};
\draw (229,212.4) node [anchor=north west][inner sep=0.75pt]    {$m'$};
\draw (361,172.4) node [anchor=north west][inner sep=0.75pt]    {$n''$};
\draw (441,182.4) node [anchor=north west][inner sep=0.75pt]    {$m''$};
\draw (147,102.4) node [anchor=north west][inner sep=0.75pt]    {$p$};
\draw (397,102.4) node [anchor=north west][inner sep=0.75pt]    {$p$};

\end{tikzpicture}
    \hfill
    \begin{tikzpicture}[x=0.4pt,y=0.4pt,yscale=-1,xscale=1]

\draw   (20,70) -- (500,70) -- (500,270) -- (20,270) -- cycle ;
\draw    (450,120.27) -- (379.65,210) ;
\draw    (310,119.73) -- (379.65,210) ;
\draw    (178,173.88) -- (145.82,212.15) ;
\draw    (108.66,160) -- (145.82,212.15) ;

\draw (180,280.4) node [anchor=north west][inner sep=0.75pt]    {$n' >p ,\ n'' >p$};
\draw (101,132.4) node [anchor=north west][inner sep=0.75pt]    {$n'$};
\draw (189,162.4) node [anchor=north west][inner sep=0.75pt]    {$m'$};
\draw (299,92.4) node [anchor=north west][inner sep=0.75pt]    {$n''$};
\draw (451,92.4) node [anchor=north west][inner sep=0.75pt]    {$m''$};
\draw (141,216.4) node [anchor=north west][inner sep=0.75pt]    {$p$};
\draw (374,211.4) node [anchor=north west][inner sep=0.75pt]    {$p$};

\end{tikzpicture}
    \\
    \centering
    \vspace{4mm}
    \begin{tikzpicture}[x=0.4pt,y=0.4pt,yscale=-1,xscale=1]

\draw   (20,70) -- (500,70) -- (500,270) -- (20,270) -- cycle ;
\draw    (120,189.26) -- (150.97,150) ;
\draw    (189.73,200.96) -- (150.97,150) ;
\draw    (430,113.78) -- (397.87,152.1) ;
\draw    (360.65,100) -- (397.87,152.1) ;

\draw (170,280.4) node [anchor=north west][inner sep=0.75pt]    {$n'< p ,\ n'' >p$};
\draw (111,192.4) node [anchor=north west][inner sep=0.75pt]    {$n'$};
\draw (191,202.4) node [anchor=north west][inner sep=0.75pt]    {$m'$};
\draw (147,122.4) node [anchor=north west][inner sep=0.75pt]    {$p$};
\draw (369,82.4) node [anchor=north west][inner sep=0.75pt]    {$n''$};
\draw (435,85) node [anchor=north west][inner sep=0.75pt]    {$m''$};
\draw (399.87,155.5) node [anchor=north west][inner sep=0.75pt]    {$p$};

\end{tikzpicture}
    \caption{Different possibilities for Case 1.2}
    \label{fig:case1 subcase2}
\end{figure}
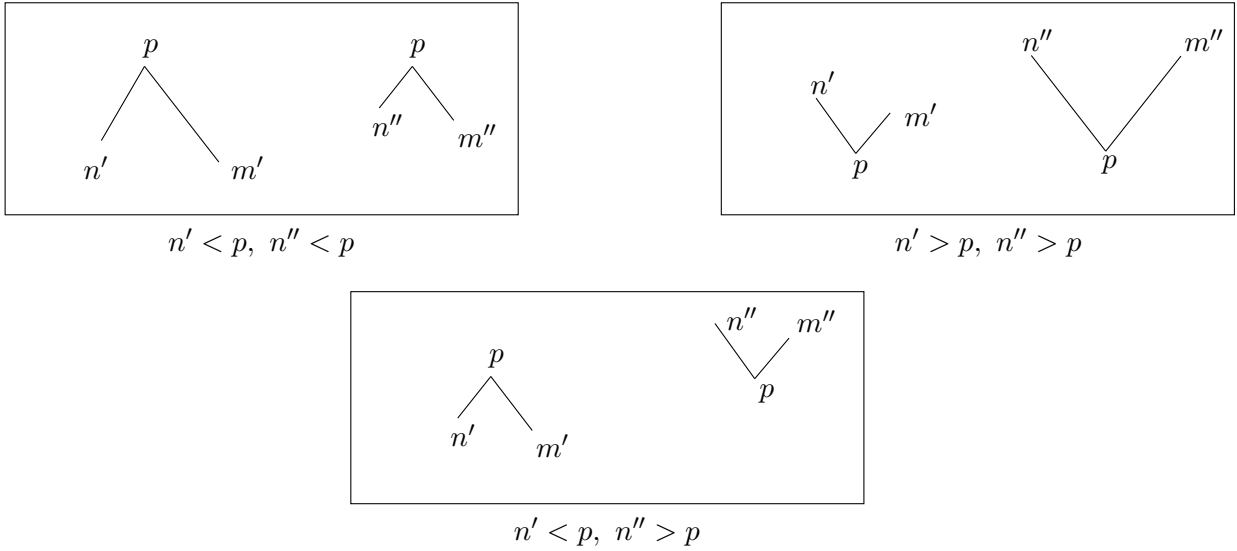

For the last possibility, assume that $n'<p$ and $n''>p$. Consequently, we have $m'<p$ and $m''>p$. Consider the zero-length string $\f u:=\winvst{p}$.

Since $\f u$ is a substring of $\winvba{\s w'}$ and $\winvba{\s w''}$, consider the first occurrence (from the right) of $\f u$ in $\winvba{\s w'}$, say that $\f u$ lies between $\alpha_{k'}$ and $\alpha_{k'+1}$ and in $\winvba{\s w''}$, say that $\f u$ lies between $\alpha_{k''}$ and $\alpha_{k''+1}$. Consider the marked strings $\dot{\f y}_1:=(\winvba{\s w'},k',k')$ and $\dot{\f y}_2:=(\winvba{\s w''},k'',k'')$. These marked strings can be extended to marked infinite strings $\INF{\dot{\f y_1}}$ and $\INF{\dot{\f y_2}}$ in an obvious way. It is straightforward to note that $\vartheta_l(\INF{\dot{\f y_1}})=-\vartheta_l(\INF{\dot{\f y_2}})=\vartheta_r(\INF{\dot{\f y_1}})=-\vartheta_r(\INF{\dot{\f y_2}})=-1$. The rest of the argument is similar to Case 1.1 above.\\

\textbf{Case 2.} There exist cyclic permutations $\s w'$ of $\s w_1^{2\cdot|\s w_2|}$ and $\s w''$ of $\s w_2^{2\cdot|\s w_1|}$ such that $\s w'=n'm'\s z'$ and $\s w''=n''m''\s z''$, where $n'<n'',\ m'<m''$, $n'<m'$ if and only if $n''<m''$, and $T_{(n',m')}\cap T_{(n'',m'')}\neq\emptyset$.

Without loss, let $n'<m'$. Consequently, $n''<m''$. Consider the least $p\in T_{(n',m')}\cap T_{(n'',m'')}$. We claim that $p=n''$. If not, then $n',n''<p$. Let $p_1$ and $p_2$ be the immediate predecessors of $p$ in $T_{(n',m')}$ and $T_{(n'',m'')}$ respectively. The choice of $p$ forces $p_1\neq p_2$. This in turn violates Condition \ref{at most one essential neighbour} of Definition \ref{defn: traced poset} since $T$ is closed subsequences, thus establishing the claim. Similarly, we can argue that the greatest element of $T_{(n',m')}\cap T_{(n'',m'')}$ is $m'$.

Since $T$ is closed under subsequences, $T_{(n'',m')}$ is a non-empty sequence in $T$. Therefore consider $\f u:=\winvst{n''m'}$. We claim that $\f u$ is a string that appears as a factor substring of $\INF{\f b_1}$ and an image substring of $\INF{\f b_2}$.

Consider the first occurrence of $\f u$ in $\winvba{\s w'}$ and in $\winvba{\s w''}$ and the appropriate marked string $\dot{\f y_1}$ and $\dot{\f y_2}$. The marked strings can be extended to $\INF{\dot{\f y_1}}$ and $\INF{\dot{\f y_2}}$ in the obvious way. We then proceed similarly as in Case 1.1.
\end{proof}

In view of Corollary \ref{cor: brick criteria}, as an immediate consequence of Lemma \ref{lem: nontrivial morphism between bands criterion in terms of WPCC}, we have the main result of the paper, which characterises band bricks in terms of weakly perfectly clustering pairs of crowns.

\main*

\begin{exmp}
Continuing from Example \ref{exmp: tracings}, consider the band $$\f b:=ba_1^{-1}a_2a_3^{-1}b^{-1}c_3c_2^{-1}c_1ba_1^{-1}a_2a_3^{-1}b^{-1}c_1^{-1}c_2c_3^{-1}.$$ Then we have $$\wba{\f b}=(v_4,1)(v_6,-1)(v_5,-1)(v_3,-1)(v_2,-1)(v_1,1)(v_5,1)(v_6,1)(v_4,-1)(v_3,-1)(v_2,-1)(v_1,1).$$ Note that $$\s w':=(v_5,-1)(v_3,-1)(v_2,-1)(v_1,1)(v_5,1)\s z'$$ and $$\s w'':=(v_4,-1)(v_3,-1)(v_2,-1)(v_1,1)(v_4,1)\s z''$$ are cyclic permutations of $\wba{\f b}^{2\cdot|\wba{\f b}|}$ for appropriate $\s z'$ and $\s z''$, where we have $(v_5,-1)<(v_4,-1)$ and $(v_5,1)<(v_4,1)$. Therefore $(\wba{\f b},\wba{\f b})$ is not perfectly clustering, which implies that $B(\f b,1,\lambda)$ is not a brick for any $\lambda\in\C K^\times$.
\end{exmp}

Here is the brief sketch of the argument of how Theorem \ref{thm: brick iff WPCC} can be obtained from Theorem \ref{thm: main} as promised at the end of \S~\ref{sec: band bricks in lamdaN}.

\begin{proof}[Proof (sketch) of Theorem \ref{thm: brick iff WPCC}]
The finite traced poset for $\Lambda_N$ contains two connected components, each of which is a linear order with $N$ elements. For a given band $\f b$ over $\Lambda_N$, $\wba{\f b}$ and $\wba{\f b^{-1}}$ will be formed by letters from distinct components. Therefore the pair $(\wba{\f b},\wba{\f b^{-1}})$ can never be weakly perfectly clustering. The rest follows from Theorem \ref{thm: main} because of Remark \ref{rmk: w is WPC iff ww is WPC}.
\end{proof}

\section{Applications}\label{sec: appl}
In this section, we mention two applications of our work. The first application characterises \emph{band semibricks}--a generalisation of semi-simple modules--in terms of weakly perfectly clustering pairs of crowns.

\begin{defn}\label{defn: semi bricks}
A \emph{semibrick} $S$ is a direct sum of finitely many bricks $B_1,\cdots,B_k$ such that there is no non-zero morphism from $B_i$ to $B_j$ for $i,j\in\{1,2,\cdots,k\}$ with $i\neq j$. A semibrick $S$ is a \emph{band semibrick} if each $B_i$ is a band brick.
\end{defn}

The following result is obtained as an immediate consequence of Lemma \ref{lem: nontrivial morphism between bands criterion in terms of WPCC} that provides a combinatorial characterisation of band semibricks.

\begin{cor}\label{cor: semi-brick criterion}
Let $\Lambda$ be a string algebra whose underlying quiver is acyclic. Let $\f b_1,\cdots,\f b_k$ be bands, $l_1,\cdots,l_k\in\mathbb N^+$ and $\lambda_1\cdots,\lambda_k\in\C K^\times$. Assume that $B(\f b_i,l_i,\lambda_i)\not\cong B(\f b_j,l_j,\lambda_j)$ for any $i\neq j$. Then the module $$S:=B(\f b_1,l_1,\lambda_1)\oplus\cdots\oplus B(\f b_k,l_k,\lambda_k)$$ is a band semi-brick if and only if $l_i=1$ for every $i\in\{1,\cdots,k\}$, and $(\wba{\f b_i},\wba{\f b_j})$ and $(\wba{\f b_i},\wba{\f b_j^{-1}})$ are weakly perfectly clustering pairs of crowns for every $i,j\in\{1,\cdots,k\}$.
\end{cor}

For the second application, recall that a special biserial algebra is a bound quiver algebra $\C KQ/\langle\rho\rangle$, whose definition is identical to that of a string algebra except that $\rho$ can contain $\C K$-linear combinations of paths of length at least 2 that share the source and target vertices. Apart from the string modules or band modules, there are only finitely many finite-dimensional indecomposables for a special biserial algebra. The following result by Mousavand and Paquette provides a characterisation of brick-infinite special biserial algebras.

\begin{thm}
\cite[Theorem~1.3]{mousavand2022biserial} For a special biserial algebra $\Lambda:=\C K Q/I$, the following are equivalent:
\begin{enumerate}
    \item $\Lambda$ is brick-infinite;
    \item for some $2\leq d\leq 2|Q_0|$, there is an infinite family $\{X_\lambda\}_{\lambda\in\C K^\times}$ of bricks with $\mathrm{dim}_\C K(X_\lambda)=d$.
\end{enumerate}
\end{thm}

Note that for string algebras (or more generally, special biserial algebras), for any $d\geq 1$, any infinite family of indecomposable modules with total dimension $d$ must contain band modules. Moreover, the set of strings, and hence the set of bands, of a fixed length is finite, and can be effectively computed. Therefore, the following combination of the above theorem with Theorem \ref{thm: main} can be used as an algorithm to determine whether a string algebra (or a special biserial algebra) $\Lambda$ with acyclic underlying quiver is brick-infinite.

\begin{cor}\label{cor: algorithm brick inf}
Let $\Lambda$ be a string algebra (or a special biserial algebra) whose underlying quiver is acyclic. Consider the (finite) set $\s B$ of bands with length at most $2|Q_0|$. Then $\Lambda$ is brick-infinite if and only if there exists a band $\f b\in \s B$ such that $B(\f b,1,\lambda)$ is a brick for some (equivalently, any) $\lambda\in\C K^\times$.
\end{cor}

\section{Comments and future directions}\label{sec: future dir}
Our treatise in this paper deals with string algebras whose underlying quiver is acyclic. Now we show that a gentle algebra whose underlying quiver is not necessarily acyclic can be ``trisected" to obtain a gentle algebra whose underlying quiver is acyclic keeping the nature (brick or not) of its band modules intact, and therefore our results are applicable.

Given a gentle algebra $\Lambda=\C K Q/\langle\rho\rangle$, where $Q=(Q_0,Q_1,s,t)$, construct its \emph{trisection}, which is again a string algebra $\Lambda^{tri}=\C KQ'/\langle\rho'\rangle$, where $Q'=(Q'_0,Q'_1,s',t')$, as follows:
$$Q'_0:=Q_0\cup\{v_\alpha^1\mid\alpha\in Q_1\}\cup\{v_\alpha^2\mid\alpha\in Q_1\},$$
$$Q'_1:=\{\alpha_1\mid\alpha\in Q_1\}\cup\{\alpha_2\mid\alpha\in Q_1\}\cup\{\alpha_3\mid\alpha\in Q_1\},$$
$$s'(\alpha_1):=s(\alpha),\ t'(\alpha_1):=v_\alpha^1,\ s'(\alpha_2):=v_\alpha^2,\ t'(\alpha_2):=v_\alpha^1,\ s'(\alpha_3):=v_\alpha^2,\ t'(\alpha_3):=t(\alpha),$$
$$\rho':=\{\alpha_1\beta_3:\alpha\beta\in\rho\}.$$

It is straightforward to note the following.
\begin{rmk}
Given a string $\gamma_k\cdots\gamma_1$ for $\Lambda$, it is a (cyclic permutation of a) band for $\Lambda$ if and only if $\gamma_{k3}\gamma_{k2}^{-1}\gamma_{k1}\cdots\gamma_{13}\gamma_{12}^{-1}\gamma_{11}$ is so too for $\Lambda^{tri}$.
\end{rmk}

The following proposition says that the nature of a band module remains intact after the modification.

\begin{prop}\label{prop:gentle to non-cycle keeps nature intact}
Let $\Lambda$ be a gentle algebra. Suppose $\f b:=\alpha_k\cdots\alpha_1$ is band for $\Lambda$ and the corresponding band for $\Lambda^{tri}$ is $\f b':=\alpha_{k3}\alpha_{k2}^{-1}\alpha_{k1}\cdots\alpha_{13}\alpha_{12}^{-1}\alpha_{11}$. Then for any $\lambda,\lambda'\in\C K^\times$, the band module $B(\f b,1,\lambda)$ is a brick if and only if the band module $B(\f b',1,\lambda')$ is also a brick.
\end{prop}
\begin{proof}
$(\Longleftarrow)$ Suppose the band module $B(\f b,1,\lambda)$ is not a brick. By Corollary \ref{cor: brick criteria}, there exists a string $\beta_l\cdots\beta_1$ which occurs as a factor substring of $\INF{\f b}$ and image substring of $\INF{\f b}$ or $\INF{(\f b^{-1})}$. It is immediate that the substring $\beta_{l3}\beta_{l2}^{-1}\beta_{l1}\cdots\beta_{13}\beta_{12}^{-1}\beta_{11}$ is a factor substring of of $\INF{\f b'}$ as well as an image substring of $\INF{\f b'}$ or $\INF{(\f {b'}^{-1})}$, which shows that the band module $B(\f b',1,\lambda')$ is not a brick.

$(\Longrightarrow)$ Suppose the band module $B(\f b',1,\lambda')$ is not a brick. Then there exists a string $\f u$ which occurs as a factor substring of $\INF{\f b'}$ and image substring of $\INF{\f b'}$ or $\INF{({\f b'}^{-1})}$. Note that the construction of $\Lambda^{tri}$ forces that the first and last syllables of $\f u$ are $\gamma_1$ and $\gamma'_3$ for some syllables $\gamma$ and $\gamma'$ of $\Lambda$. Consequently, $\f u=\beta_{l3}\beta_{l2}^{-1}\beta_{l1}\cdots\beta_{13}\beta_{12}^{-1}\beta_{11}$ for some syllables $\beta_1,\cdots,\beta_l$ of $\Lambda$, which again implies that $\beta_l\cdots\beta_1$ is a factor substring of $\INF{\f b}$ and an image substring of $\INF{\f b}$ or $\INF{(\f b^{-1})}$.
\end{proof}

If $\Lambda$ is not gentle, then the above modification does not keep the collection of bands (up to cyclic permutation) intact as the next example demonstrates.
\begin{exmp}
If we trisect $GP_{2,3}$ (Figure \ref{fig: GP23}), then we get the string algebra $\Gamma'$ from Figure \ref{fig:motivation of cov quiv}. Note that $b^{-1}$ is not (a cyclic permutation of) a band for $GP_{2,3}$ but $b_3^{-1}b_2b_1^{-1}$ is  (a cyclic permutation of) a band for the $\Gamma'$.
\end{exmp}

\begin{figure}[h]
    \begin{tikzcd}
v \arrow["b"', loop, distance=2em, in=35, out=325] \arrow["a"', loop, distance=2em, in=215, out=145]
\end{tikzcd}
    \caption{GP$_{2,3}$ with $\rho=\{a^2,b^3,ab,ba\}$}
    \label{fig: GP23}
\end{figure}
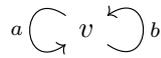

At this end, notice that due to the possibility of recovering a string algebra with acyclic underlying quiver from the associated traced poset (see \S~\ref{sec:construct string alg}), it is not possible to generalise the methods of this paper to include all string algebras. Therefore, it is natural to ask the following.

\begin{ques}
Is it possible to develop an algorithmic test (preferably using word combinatorics) for determining whether a given band module is a brick for any string algebra?
\end{ques}

Moving on to the other side, recall from Definition \ref{defn: weak perf clus} that weakly perfectly clustering words on a linearly ordered alphabet $\s A$ are a generalisation of perfectly clustering words on $\s A$. The latter class was shown in \cite{lapointebook} to be the class of positive primitive elements of the free group $F_{\s A}$ generated by $\s A$, where primitivity refers to elements of some basis. Such words were also generated by endomorphisms of $F_{\s A}$. We would like to ask if equivalent versions of such results are possible for weakly perfectly clustering words on $\s A$.

\begin{ques}
What is the subset of $F_{\s A}$ that corresponds to the set of weakly perfectly clustering words on $\s A$? Is it possible to generate those words using some endomorphisms of the free group $F_{\s A}$?
\end{ques}





\bibliographystyle{alpha}
\bibliography{main}
\end{document}